\input ./style/arxiv-vmsta.cfg
\documentclass[numbers,compress,v1.0.1]{vmsta}
\usepackage{vtexbibtags}

\volume{5}
\issue{2}
\pubyear{2018}
\firstpage{167}
\lastpage{189}
\aid{VMSTA101}
\doi{10.15559/18-VMSTA101}
\articletype{research-article}

\startlocaldefs

\allowdisplaybreaks

\newtheorem{theorem}{Theorem}
\newtheorem{lemma}{Lemma}

\theoremstyle{definition}
\newtheorem{remark}{Remark}

\newcommand{\E}{\mathsf{E}}

\theoremstyle{definition}

\newcolumntype{d}[1]{D{.}{.}{#1}}

\ifdefined\HCode 
\def\texbold#1{#1}
\else 
\def\texbold#1{\textbf{#1}}
\fi
\endlocaldefs

\begin{document}

\begin{frontmatter}
\pretitle{Research Article}

\title{Properties of Poisson processes directed by compound
Poisson-Gamma subordinators}


\author{\inits{K.}\fnms{Khrystyna}~\snm{Buchak}\ead
[label=e1]{kristina.kobilich@gmail.com}}
\author{\inits{L.}\fnms{Lyudmyla}~\snm{Sakhno}\thanksref{cor1}\ead
[label=e2]{lms@univ.kiev.ua}}
\thankstext[type=corresp,id=cor1]{Corresponding author.}

\address{Mechanics and Mathematics Faculty,\break
\institution{Taras Shevchenko National University of Kyiv},\break
Volodymyrska 64/11,\break
01601 Kyiv, \cny{Ukraine}}

\markboth{K. Buchak, L. Sakhno}{Properties of Poisson processes
directed by compound Poisson-Gamma subordinators}



\begin{abstract}
In the paper we consider time-changed Poisson processes where the time
is expressed by compound Poisson-Gamma subordinators $G(N(t))$ and
derive the expressions for their hitting times. We also study the
time-changed Poisson processes where the role of time is played by the
processes of the form $G(N(t)+at)$ and by the iteration of such processes.
\end{abstract}


\received{\sday{12} \smonth{1} \syear{2018}}
\revised{\sday{7} \smonth{4} \syear{2018}}
\accepted{\sday{11} \smonth{4} \syear{2018}}
\publishedonline{\sday{2} \smonth{5} \syear{2018}}
\end{frontmatter}

\section{Introduction}\label{sec1}

Poisson processes with randomized time have been intensively studied in
the recent literature. The most popular models of such processes are
represented by the space-fractional and time-fractional Poisson
processes where a random time-change is introduced by a stable
subordinator or its inverse correspondingly (we refer, for example, to
\cite{GOS,OP,OT,B-D'O,LST,ALM,OP2013}, to mention only few, see also
references therein).

In the paper \cite{OT} a general class of time-changed Poisson
processes $N^f(t)=N(H^{f}(t))$, $t>0$, has been introduced and studied,
where $
N(t)$ is a Poisson process and $H^{f}(t)$ is an arbitrary subordinator with
Laplace exponent $f$, independent of $N(t)$. Distributional properties,
hitting times and governing equations for such processes were presented
in \cite{OT,GOS}; the case of iterated time change and some further
generalizations of the class of process $N^f(t)$ were also considered
in \cite{GOS}. Hitting times for the iterated Poisson process were
studied in \cite{OP2,D}, and for population processes with random time
in \cite{B-O}.

In the papers \cite{BS,KS} time-changed Poisson processes were studied
for the case where the role of time is played by compound Poisson-Gamma
subordinators and their inverse processes.
In the present paper, we continue to investigate the properties of the
processes $N(G_N(t))$, $t>0$, where $G_N(t)=G(N(t))$ is the compound
Poisson-Gamma process. Some motivation for considering this class of
processes is given in \cite{BS} (see Section~\ref{sec3} therein). We derive the
expressions for the hitting times and first passage times of the
processes $N(G_N(t))$ in Sections~\ref{sec3} and \ref{sec4}. We next study in Section~\ref{sec5}
the time-change introduced by the process $G_{N+a}(t)=G(N(t)+at)$,
where the Gamma process $G(t)$ and the Poisson process with a drift
$N(t)+at$ are independent. In Section~\ref{sec6} we consider some kinds of
iterated Bessel transforms and their use for time-change in the Poisson process.

\section{Preliminaries}\label{sec2}

In this section we recall some results on time-changed Poisson
processes, which will be used in the next sections.
In the paper \cite{OT} the time-changed Poisson processes $N^f(t)=N
(H^f(t) )$, $t>0$, have been studied, where
$N(t)$ is the Poisson process with intensity $\lambda$ and $H^f(t)$ is
the subordinator with a Bern\v{s}tein function $f(u)$, independent of
$N(t)$. Their distributions are characterized as follows:
%
%
\begin{equation}
\label{Pdt} Pr \bigl\{N^f[t,t+dt)=k \bigr\}= %
\begin{cases}
dt\frac{\lambda^k}{k!}\int_{0}^{\infty}e^{-\lambda s}s^k \nu(ds)+o(dt),
k\geq1,\\
1-dt\int_{0}^{\infty} (1-e^{-\lambda s} ) \nu(ds)+o(dt), k=0,\\
\end{cases}
\end{equation}
\vspace*{-6pt}
%
\begin{equation}
\label{pkf} p_k^f(t)=P \bigl\{N^f(t)=k
\bigr\}=\frac{(-1)^k}{k!}\frac
{d^k}{du^k}e^{-tf(\lambda u)}\bigg\arrowvert_{u=1},
\end{equation}
the probability generating function of $N^f(t)$ is given by
%
%
\begin{equation}
\label{Gf} G^f(u,t)=e^{-tf(\lambda(1-u))},\quad  |u|\leq1.
\end{equation}

The time-changed Poisson processes $N^f(t)$, $t>0$, have independent
stationary increments (see, e.g., the general result on subordinated L\'
{e}vy processes given in Theorem 1.3.25 \cite{A}).

It is also shown in \cite{OT} that the probabilities of the processes
$N^f(t)$ satisfy the difference-differential equations:
%
%
\begin{equation}
\label{dtpkf} \frac{d}{dt}p_k^f(t)=-f(
\lambda)p_k^f(t)+\sum_{m=1}^{k}
\frac{\lambda
^m}{m!}p_{k-m}^f(t)\int_{0}^{\infty}e^{-s \lambda}s^m
\nu(ds), \quad  k\geq 0, \, t>0,
\end{equation}
with the usual initial conditions: $p_0^f(0)=1$, and $p_k^f(0)=0$ for
$k\ge1$. The equation \eqref{dtpkf} can be also written in the
following form (see \cite{OT}, Remark 2.3):
%
%
\begin{equation}
\label{dtpkf1} \frac{d}{dt}p_k^f(t)=-f\bigl(
\lambda(I-B)\bigr)p_k^f(t), \quad  k\geq0, \, t>0,
\end{equation}
where $B$ is the shift operator: $B p_k^f(t)=p_{k-1}^f(t)$, and it is
supposed that $p_{-1}(t)=0$.

Let $N_{1}(t)$ be the Poisson process with intensity $\lambda_{1}$,
and let $G_N(t)=G(N(t))$, $t>0$, be the compound Poisson-Gamma
subordinator with parameters $\lambda, \alpha, \beta$, that is, with
the Laplace exponent $f(u)=\lambda\beta^\alpha (\beta^{-\alpha
}-(\beta+u)^{-\alpha} )$ and the L\'evy measure $\nu(du)=\lambda
\beta^\alpha (\varGamma(\alpha) )^{-1}u^{\alpha-1}e^{-\beta u}
du$, $\lambda, \alpha, \beta>0$. In the case when $\alpha=1$ we have
the compound Poisson-exponential subordinator, which we will denote as $E_N(t)$.

Consider the time-changed process $N_{1}(G_{N}(t))=N_{1}(G(N(t)))$,
$t>0$, where the compound Poisson-Gamma process $
G_{N}(t)$ is independent of $N_{1}(t)$.

We recall the results on probability distributions of the process
$N_{1}(G_{N}(t))$, which were presented in our previous paper \cite{BS}
and will be used in the next sections.

%
\begin{theorem}\cite{BS}
Probability mass function of the process $X(t)=N_{1}(G_{N}(t))$,
$t>0$, is
given by
%
%
\begin{equation}
p_{k}(t)=P\bigl(X(t)=k\bigr)=\frac{e^{-\lambda t}}{k!}\frac{\lambda
_{1}^{k}}{(\lambda_{1}+\beta)^{k}}\sum
_{n=1}^{\infty
}\frac{ ( \lambda t\beta^{\alpha} ) ^{n}\varGamma(\alpha
n+k)}{(\lambda_{1}+\beta)^{\alpha n}n!\varGamma(\alpha n)} \quad \mbox{\text{for}}\ k\ge1 \label{pkN1GN},
\end{equation}
and
%
%
\begin{equation}
p_{0}(t)=\exp\bigl\{-tf(\lambda_1)\bigr\}=\exp \biggl
\{-\lambda t \biggl(1-\frac
{\beta^{\alpha
}}{(\lambda_{1}+\beta)^{\alpha}} \biggr) \biggr\}\label{p0NGN}.
\end{equation}
The probabilities $p_{k}(t)$, $k\geq0$, satisfy the following
system of difference-differential equations:
%
%
\begin{equation}
\frac{d}{dt}p_{k}(t)= \biggl( \frac{\lambda\beta^{\alpha}}{(\lambda
_{1}+\beta)^{\alpha}}-\lambda \biggr)
p_{k}(t)+\frac{\lambda\beta
^{\alpha}}{(\lambda_{1}+\beta)^{\alpha}}\sum_{m=1}^{k}
\frac{\lambda
_{1}^{m}}{ ( \lambda_{1}+\beta ) ^{m}}\frac{\varGamma
(m+\alpha)}{%
m!\varGamma(\alpha)}p_{k-m}(t).\label{dpkNGN}
\end{equation}
\end{theorem}
%
%
\begin{remark}
Probability mass function of the process $X(t)=N_{1}(G_{N}(t))$ for
$k\geq1$ can be represented with the use of the generalized Wright
function in the following form:
%
%
\begin{equation}
p_{k}(t)=P\bigl(X(t)=k\bigr)= \frac{e^{-\lambda t}}{k!}\frac{\lambda
_{1}^{k}}{(\lambda_{1}+\beta)^{k}}\,\,
{_1\varPsi_{1}} \biggl((k,\alpha),(0,\alpha),
\frac{\lambda t \beta^{\alpha
}}{(\lambda_1+\beta)^{\alpha}} \biggr),
\end{equation}
where
%
%
\begin{equation}
_{p}\varPsi_{q}\bigl((a_i,
\alpha_i), (b_j, \beta_j), z\bigr)=\sum
_{k=0}^{\infty
}\frac{\prod\limits_{i = 1}^{p}\varGamma(a_i+\alpha_i
k)}{\prod\limits_{j = 1}^{q}\varGamma(b_j+\beta_j k)}
\frac{z^k}{k!}
\end{equation}
is the generalized Wright function defined for $z\in C$, $a_i, b_i\in
C$, $\alpha_i, \beta_i\in R$, $\alpha_i, \beta_i\neq0$ and $\sum\alpha
_i - \sum\beta_i > -1$ (see, e.g., \cite{HMS}).
\end{remark}

To represent the distribution function in the next theorem we will use
the three-parameter generalized Mittag-Leffler function, which is
defined as follows:
%
%
\begin{equation}
\mathcal{E}_{\rho,\delta}^{\gamma}(z)=\sum_{k=0}^{\infty}
\frac
{\varGamma
(\gamma+k)}{\varGamma(\gamma)}\frac{z^{k}}{k!\varGamma(\rho k+\delta
)},\quad z\in C,\text{ }\rho,\delta,\gamma\in
C, \label{ML3}
\end{equation}
with $\text{ Re}(\rho)>0,\text{Re}(\delta)>0,\text{Re}(\gamma)>0$
(see, e.g., \cite{HMS}).

%
\begin{theorem}\cite{BS}
Let $X(t)=N_{1}(E_{N}(t))$. Then for $k\ge1$
%
%
\begin{equation}
p_{k}^{E}(t)=P\bigl(X(t)=k\bigr)=e^{-\lambda t}
\frac{\lambda_{1}^{k}\lambda
\beta t}{%
(\lambda_{1}+\beta)^{k+1}}\mathcal{E}_{1,2}^{k+1} \biggl(
\frac{\lambda\beta t}{\lambda_{1}+\beta} \biggr) \label{pkNEN};
\end{equation}
\vspace*{-6pt}
%
\begin{equation}
p_{0}(t)=\exp \biggl\{- \frac{\lambda\lambda_1 t }{\lambda
_{1}+\beta} \biggr\}\label{p0NEN},
\end{equation}
and the probabilities $p_{k}(t)$, $k\ge0$, satisfy the following
equation:
%
%
\begin{equation}
\frac{d}{dt}p_{k}^{E}(t)=-\lambda
\frac{\lambda_{1}}{\lambda
_{1}+\beta}%
p_{k}^{E}(t)+\frac{\lambda\beta}{\lambda_{1}+\beta}
\sum_{m=1}^{k} \biggl( \frac{\lambda_{1}}{\lambda_{1}+\beta}
\biggr) ^{m}p_{k-m}^{E}(t). \label{dpkNEN}
\end{equation}
\end{theorem}
Figures \ref{figLefta} and \ref{figRighta} show the behavior of the
probabilities (\ref{pkN1GN}) and (\ref{pkNEN}), for various choices of
$t$ $ (t=1,2,3)$.

%
\begin{figure}
\includegraphics{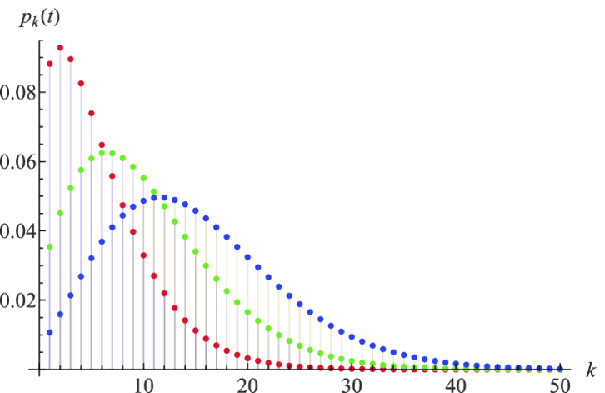}
\caption{Probabilities (\ref{pkN1GN}), for values of $\alpha=2$, $\beta
=0.8$, $ \lambda=2$, $\lambda_1=1$}
\label{figLefta}
\end{figure}

%
\begin{figure}
\includegraphics{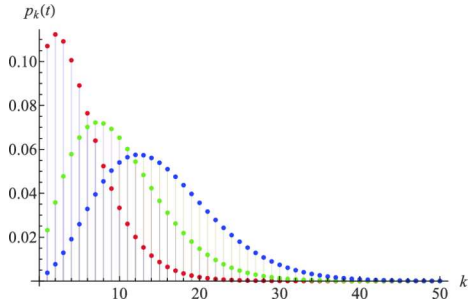}
\caption{Probabilities (\ref{pkNEN}), for values of $\beta=0.8, \lambda
=4, \lambda_1=1$}
\label{figRighta}
\end{figure}

\section{Hitting times of the subordinated Poisson process $N_1(G_N(s))$}\label{sec3}

In this section we study the hitting times for the process
$N_1(G_N(s))$ defined as
\[
T_k=\inf\bigl\{s:N_1 \bigl(G_N(s)
\bigr)=k\bigr\},\quad  k\geq1.
\]

In the next theorem we obtain the analytic expression for $P \{
T_k<\infty \}$.

%
\begin{theorem}
For the random times $T_k$ we have that
%
%
\begin{eqnarray}
P \{T_k<\infty \}=\frac{ \lambda_1^k}{k! (\lambda_1+\beta
 )^{k}} \biggl(1- \biggl(
\frac{\beta}{\lambda_{1}+\beta} \biggr)^{\alpha
} \biggr)\sum
_{n=1}^{\infty} \biggl(\frac{\beta}{\lambda_1+\beta}
\biggr)^{\alpha n}\frac{\varGamma(\alpha n+k)}{\varGamma(\alpha n)}. \label{PTkNGN}
\end{eqnarray}
\end{theorem}

\begin{proof}
Following the first lines of the proof of Theorem 2.3. from \cite
{GOS}, using the independence of increments of the process
$N_1(G_N(s))$, we write the lines \eqref{14} and \eqref{15} below, and
then, having in mind
(\ref{Pdt}), we come to the expression \eqref{dPTkNGN}:
%
%
\begin{gather}
P \{T_k\in ds \}=P \Biggl\{\bigcup_{j=1}^{k}
\bigl\{ N_1\bigl(G_N(s)\bigr)=k-j, N_1
\bigl(G_N[s,s+ds\bigr))=j \bigr\} \Biggr\}\label{14}
\\
\quad\quad\quad\quad\quad=\sum_{j=1}^{k}
P \bigl\{N_1\bigl(G_N(s)\bigr)=k-j \bigr\}P \bigl
\{N_1\bigl(G_N[s,s+ds\bigr))=j \bigr\} \label{15}
\\
=\sum_{j=1}^{k} p_{k-j}(s)
\frac{\lambda_{1}^j}{j!}\int_{0}^{\infty
}e^{-\lambda_{1} t}t^j
\nu(dt)ds.\quad\quad\quad\quad\qquad\label{dPTkNGN}
\end{gather}
Now we note that the expression \eqref{dPTkNGN} coincides with the
second term in the r.h.s. of the equation \eqref{dtpkf}, and,
therefore, we can write:
\begin{align}
\label{PTkdsNGN}
P \{T_k\in ds \}={}& \biggl\{\frac{d}{ds}p_k(s)+f(
\lambda _{1})p_k(s) \biggr\}ds
\nonumber
\\
={}& \Biggl\{\frac{1}{k!}\frac{\lambda
_{1}^{k}}{(\lambda_{1}+\beta)^{k}}\sum
_{n=1}^{\infty
}\frac{ ( \lambda\beta^{\alpha} ) ^{n}\varGamma(\alpha
n+k)}{(\lambda_{1}+\beta)^{\alpha n}n!\varGamma(\alpha n)}\frac
{d}{ds}e^{-\lambda s}
s^n
\nonumber
\\
&{}+ \biggl(\lambda-\frac{\lambda\beta^{\alpha}}{ (\lambda
_{1}+\beta )^{\alpha}} \biggr)p_k(s) \Biggr\}
ds
\nonumber
\\
={}&\frac{e^{-\lambda s}\lambda
_{1}^{k}}{k!(\lambda_{1}+\beta)^{k}}\sum_{n=1}^{\infty
}
\frac{ ( \lambda\beta^{\alpha} ) ^{n}\varGamma(\alpha
n+k)}{(\lambda_{1}+\beta)^{\alpha n}n!\varGamma(\alpha n)} \biggl(-\lambda s^n+n s^{n-1}
\nonumber
\\
&{}+s^n \biggl(\lambda-\frac{\lambda\beta^{\alpha}}{ (\lambda
_{1}+\beta )^{\alpha}} \biggr) \biggr)ds
\nonumber
\\
={}&\frac{e^{-\lambda s}\lambda
_{1}^{k}}{k!(\lambda_{1}+\beta)^{k}}\sum_{n=1}^{\infty
}
\frac{ ( \lambda\beta^{\alpha} ) ^{n}\varGamma(\alpha
n+k)}{(\lambda_{1}+\beta)^{\alpha n}n!\varGamma(\alpha n)} \biggl(n s^{n-1}-s^n\frac{\lambda\beta^{\alpha}}{ (\lambda_{1}+\beta
)^{\alpha}}
\biggr)ds.\nonumber\\
\end{align}
Finally, from the expression (\ref{PTkdsNGN}) we obtain:
\begin{align*}
P \{T_k<\infty \}={}&\frac{\lambda
_{1}^{k}}{k!(\lambda_{1}+\beta)^{k}}\sum
_{n=1}^{\infty
}\biggl\{\frac{ ( \lambda\beta^{\alpha} ) ^{n}\varGamma(\alpha
n+k)}{(\lambda_{1}+\beta)^{\alpha n}n!\varGamma(\alpha n)}
\nonumber
\\
& {}\times\int
_{0}^{\infty}e^{-\lambda s}\biggl(n s^{n-1}-s^n\frac{\lambda\beta^{\alpha}}{ (\lambda_{1}+\beta
)^{\alpha}} \biggr)ds\biggr\}
\nonumber
\\
={}&\frac{\lambda
_{1}^{k}}{k!(\lambda_{1}+\beta)^{k}}\sum_{n=1}^{\infty
}\biggl\{
\frac{ ( \lambda\beta^{\alpha} ) ^{n}\varGamma(\alpha
n+k)}{(\lambda_{1}+\beta)^{\alpha n}n!\varGamma(\alpha n)}
\nonumber
\\
& {}\times\biggl(n \frac
{\varGamma(n)}{\lambda^n}-\frac{\lambda\beta^{\alpha}}{ (\lambda_{1}+\beta
)^{\alpha}}\frac{\varGamma(n+1)}{\lambda^{n+1}} \biggr)\biggr\}
\nonumber
\\
={}&\frac{ \lambda_1^k}{k! (\lambda_1+\beta )^{k}} \biggl(1- \biggl(\frac{\beta}{\lambda_{1}+\beta} \biggr)^{\alpha}
\biggr)\sum_{n=1}^{\infty} \biggl(
\frac{\beta}{\lambda_1+\beta} \biggr)^{\alpha
n}\frac{\varGamma(\alpha n+k)}{\varGamma(\alpha n)}.\quad
\qedhere
\end{align*}
\end{proof}

%
\begin{remark}
We note that the expression for the hitting times \eqref{PTkNGN} does
not depend on the parameter $\lambda$, the intensity parameter of the
inner Poisson process involved in the time-changed process
\[
N_1\bigl(G_N(s)\bigr)=N_1(G\bigl(N(t)
\bigr)=N_{\lambda_1}\bigl(G_{(\alpha, \beta)}\bigl(N_\lambda(t)\bigr)
\bigr).
\]
\end{remark}
In the paper \cite{GOS} the authors show that in general case for the
process $N^f(t)=N (H^f(t) )$ the probabilities $P \{
T_k^f\in ds \}$ can be represented as follows (see the formula
(2.8) in \cite{GOS}):
%
%
\begin{equation}
P \bigl\{T_k^f\in ds \bigr\}=\sum
_{j=0}^{k-1}\frac{(-1)^j}{j!}\frac
{d^j}{du^j}e^{-sf(\lambda u)}
\bigg\arrowvert_{u=1} ds\frac{\lambda
^{k-j}}{(k-j)!}\int_{0}^{\infty}e^{-\lambda t}t^{k-j}
\nu(dt)\label{GOS2.8}
\end{equation}
To be able to use the above formula, it is necessary to calculate the
derivatives
%
\begin{equation}
\frac{d^j}{du^j}\frac{1}{f(\lambda_1 u)} \arrowvert _{u=1}.\label{der}
\end{equation}
It is noted in \cite{GOS} that evaluation of these derivatives seems
possible only for a small subset of Bern\v{s}tein functions. One
example of such functions is $f(u)=u^\alpha$, the Bern\v{s}tein
function which corresponds to the stable subordinator. We show below
another example of such functions.

In the case when $\alpha=1$ the process $G_{N}(t)$
becomes $E_{N}(t)$, the compound Poisson process with
exponentially distributed jumps that has Laplace exponent $f(u)=\lambda
\frac{u}{\beta+u}, \lambda>0, \beta>0$. For this function it is
possible to calculate the derivatives \eqref{der}, since we easily find:
%
%
\begin{equation}
\frac{d^j}{du^j}\frac{1}{f(\lambda_1 u)}= (-1)^j j!
\frac{\beta
}{\lambda\lambda_1}u^{-(j+1)},\quad  j\geq1 , \qquad  \frac{1}{f(\lambda_1 u)}=
\frac{\beta+\lambda_1 u}{\lambda\lambda_1 u},\quad  j=0.\label{dj}
\end{equation}
Therefore, we can use the formula \eqref{GOS2.8} for the process $N_1(E_N(T))$.

%
\begin{theorem}
For the random times
\[
T_k^E=\inf\bigl\{s:N_1
\bigl(E_N(s) \bigr)=k\bigr\}, \quad k\geq1,
\]
we have that
%
%
\begin{equation}
P \bigl\{T_k^E<\infty \bigr\}=\frac{\beta}{\lambda_1+\beta}.
\label{PTkNEN}
\end{equation}
\end{theorem}

\begin{proof}
Using the formulas \eqref{GOS2.8} and \eqref{dj}, we obtain:
%
%
\begin{eqnarray}
P \bigl\{T_k^E\in ds \bigr\}&=&\sum
_{j=0}^{k-1}\frac{(-1)^j}{j!}\frac
{d^j}{du^j}e^{-sf(\lambda_1 u)}
\bigg\arrowvert_{u=1} ds\frac{\lambda
_1^{k-j}}{(k-j)!}
\nonumber
\\
&&{} \times\int_{0}^{\infty}e^{-\lambda_1 t}t^{k-j}
\lambda\beta e^{-\beta t} dt
\nonumber
\\
&=&\frac{\lambda\beta}{\lambda_1+\beta}\sum_{j=0}^{k-1}
\frac
{(-1)^j}{j!}\frac{\lambda_1^{k-j}}{(\lambda_1+\beta)^{k-j}}\frac
{d^j}{du^j}e^{-sf(\lambda_1 u)}
\bigg\arrowvert_{u=1}ds. \label{PTkdsNEN}\quad
\end{eqnarray}
Therefore,
\begin{align*}
P \bigl\{T_k^E<\infty \bigr\}&=\frac{\lambda\beta}{\lambda_1+\beta
}\sum
_{j=1}^{k-1}\frac{(-1)^j}{j!} \biggl(
\frac{\lambda_1}{\lambda_1+\beta
} \biggr)^{k-j}\frac{d^j}{du^j}\frac{1}{f(\lambda_1 u)}
\bigg\arrowvert _{u=1}
\nonumber
\\
+\frac{\beta\lambda_1^{k-1}}{(\lambda_1+\beta)^k}&=\frac{\lambda
\beta}{\lambda_1+\beta}\sum_{j=1}^{k-1}
\frac{(-1)^j}{j!} \biggl(\frac
{\lambda_1}{\lambda_1+\beta} \biggr)^{k-j}(-1)^j
j! \frac{\beta}{\lambda
\lambda_1}+\frac{\beta\lambda_1^{k-1}}{(\lambda_1+\beta)^k}
\\
&=\frac{\lambda_1^{k-1} \beta^2 }{(\lambda_1+\beta)^{k+1}}\sum_{j=1}^{k-1}
\biggl(\frac{\lambda_1+\beta}{\lambda_1} \biggr)^{j}+\frac
{\beta\lambda_1^{k-1}}{(\lambda_1+\beta)^k}
\\
&= \frac{ \beta}{\lambda_1+\beta} \biggl(1- \biggl(\frac{\lambda_1}{\beta
+\lambda_1} \biggr)^{k-1}
\biggr)+\frac{\beta\lambda_1^{k-1}}{(\lambda
_1+\beta)^k}=\frac{\beta}{\lambda_1+\beta}.
\qedhere
\end{align*}
\end{proof}

\begin{remark}
For all Bern\v{s}tein functions $f(u)$ it holds (see \cite{GOS}):
\[
P \bigl\{T_1^f<\infty \bigr\}<1.
\]
For our case when
$f(u)=\lambda\beta^{\alpha} (\beta^{-\alpha}- (\beta+u
)^{-\alpha} )$,
we obtain:
\[
P \{T_1<\infty \}=\frac{\lambda_1}{\lambda_1+\beta}\frac{\alpha
\beta^{\alpha}}{(\lambda_1+\beta)^{\alpha}-\beta^{\alpha}},
\]
and when
$f(u)=\lambda\frac{u}{\beta+u}, \lambda>0, \beta>0$,
we have that for all $k\geq1$:
\[
P \bigl\{T_k^E<\infty \bigr\}=\frac{\beta}{\lambda_1+\beta}<1.
\]
\end{remark}

%
\begin{remark}
It is easy to calculate the expression $P \{T_k<\infty \}$
for the case of process $N_1(G_{N}(t))$ with the parameter $\alpha=2$:
%
%
\begin{eqnarray}
P \{T_k<\infty \}=\frac{ 1}{k!} \biggl(1- \biggl(
\frac{\beta
}{\lambda_{1}+\beta} \biggr)^{2} \biggr) \biggl[\frac{\beta}{2}
\biggl(\frac
{1}{\lambda_{1}}+\frac{\lambda_1^k}{ (\lambda_{1}+2\beta
)^{k+1}} \biggr) \biggr].
\end{eqnarray}
Indeed, we notice that the sum in the r.h.s. of the formula \eqref
{PTkNGN} can be obtained by summing the even terms of the following sum:
\begin{eqnarray*}
\sum_{n=0}^{\infty}x^n
\frac{\varGamma(n+k)}{k!\varGamma(n)}=x (1-x )^{-(k+1)},
\end{eqnarray*}
where $|x|<1$.
\end{remark}

%
\begin{lemma}
The expression for the probabilities $ P \{T_k^E\in ds \}$
for the process\break$N_1(E_N(t))=N_{\lambda_1}(E_\beta(N_\lambda(t))$
can be written in the following form:
\[
P \bigl\{T_k^E\in ds \bigr\}=\frac{\lambda\beta}{\lambda_1+\beta}\sum
_{n=0}^{\infty}P \bigl\{N_{\lambda}(s)=n
\bigr\}P \bigl\{N_{\beta} \bigl(E_{\lambda_1}(k) \bigr)=n \bigr\}ds,
\]
here $ P  \{N_ {\beta}  (E_ {\lambda_1} (k)  ) = n
\} $ is the distribution of the time-changed Poisson process, where the
role of time is played by the exponential process, subscripts denote
the parameters of the processes.
\end{lemma}
\begin{proof}
Let us return to the formula
\eqref{dPTkNGN}. Having in mind \eqref{Pdt} and \eqref{pkN1GN},
we come to the following:
%
%
\begin{align}
&P \{T_k\in ds \}
\nonumber\\
&\quad =\sum_{j=1}^{k-1}
\frac{e^{-\lambda s}
\lambda_1^{k-j}}{(k-j)!(\lambda_1+\beta)^{k-j}}\sum_{n=1}^{\infty} \biggl(
\frac{\lambda s \beta^{\alpha}}{ (\lambda_1+\beta )^{\alpha
}} \biggr)^n\frac{\varGamma(\alpha n +k-j)}{n!\varGamma(\alpha n)}
\nonumber
\\
&{}\qquad  \times\frac{\lambda_1^j}{j!}\int_{0}^{\infty}e^{-\lambda_1
t}t^j
\lambda\beta^{\alpha} \bigl(\varGamma(\alpha) \bigr)^{-1}t^{\alpha
-1}e^{-\beta t}
dtds
\nonumber
\\
&\qquad {}+\exp \biggl\{-\lambda s \biggl(1-\frac{\beta^{\alpha
}}{(\lambda_{1}+\beta)^{\alpha}} \biggr) \biggr\}
\frac{\lambda
_1^k}{k!}\int_{0}^{\infty}e^{-\lambda_1 t}t^k
\lambda\beta^{\alpha} \bigl(\varGamma(\alpha) \bigr)^{-1}t^{\alpha-1}e^{-\beta t}
dtds
\nonumber
\\
&\quad =\frac{\lambda\beta^{\alpha}e^{-\lambda s}}{\varGamma(\alpha)}\sum_{j=1}^{k-1}
\frac{ \lambda_1^{k-j}}{(k-j)!(\lambda_1+\beta)^{k-j}}\frac
{\lambda_1^j}{j!}\frac{\varGamma(j+\alpha)}{ (\lambda_1+\beta
)^{\alpha+j}}\sum
_{n=1}^{\infty} \biggl(\frac{\lambda s \beta^{\alpha
}}{ (\lambda_1+\beta )^{\alpha}}
\biggr)^n
\nonumber
\\
&\qquad {} \times\frac{\varGamma(\alpha n \,{+}\,k\,{-}\,j)}{n!\varGamma(\alpha n)}ds
+ \exp \biggl\{-\lambda s \biggl(1-
\frac{\beta^{\alpha
}}{(\lambda_{1}\,{+}\,\beta)^{\alpha}} \biggr) \biggr\}\frac{\lambda\lambda
_1^k\beta^{\alpha}}{(\lambda_1+\beta)^{\alpha+k}}\frac{\varGamma(\alpha
+k)}{k!\varGamma(\alpha)}ds
\nonumber
\\
&\quad =\frac{\lambda\beta^{\alpha}e^{-\lambda s}}{\varGamma(\alpha)}\frac
{\lambda_1^k}{ (\lambda_1+\beta )^{\alpha+k}}\sum_{j=1}^{k-1}
\frac{ \varGamma(j+\alpha)}{(k-j)!j!}\sum_{n=1}^{\infty} \biggl(
\frac{\lambda s \beta^{\alpha}}{ (\lambda_1+\beta )^{\alpha
}} \biggr)^n\frac{\varGamma(\alpha n +k-j)}{n!\varGamma(\alpha n)}ds
\nonumber
\\
&\qquad {}+\exp \biggl\{-\lambda s \biggl(1-\frac{\beta^{\alpha
}}{(\lambda_{1}+\beta)^{\alpha}} \biggr) \biggr\}
\frac{\lambda\lambda
_1^k\beta^{\alpha}}{(\lambda_1+\beta)^{\alpha+k}}\frac{\varGamma(\alpha
+k)}{k!\varGamma(\alpha)}ds. \label{2dPTkNGN}
\end{align}

Now we take $ \alpha= 1 $ in (\ref{2dPTkNGN}) and obtain:
%
%
\begin{eqnarray}
P \bigl\{T_k^E\in ds \bigr\}&=&\frac{\lambda\beta\lambda_1^k
e^{-\lambda s}}{ (\lambda_1+\beta )^{1+k}}\sum
_{j=1}^{k-1}\frac
{ \varGamma(j+1)}{(k-j)!j!}\sum
_{n=1}^{\infty} \biggl(\frac{\lambda s \beta
}{\lambda_1+\beta}
\biggr)^n\frac{\varGamma(n +k-j)}{n!\varGamma(
n)}ds
\nonumber
\\[3pt]
&&{}+\exp \biggl\{-\lambda s \!\biggl(\!1-\frac{\beta}{\lambda_{1}+\beta} \!\biggr) \!\biggr\}
\frac{\lambda\lambda_1^k\beta}{(\lambda_1+\beta
)^{k+1}}\frac{\varGamma(k+1)}{k!}ds= I_1+I_2
\nonumber
\\[3pt]
&=&
\frac{\lambda\beta\lambda_1^k e^{-\lambda s}}{ (\lambda
_1+\beta )^{1+k}}\sum_{n=1}^{\infty} \biggl(
\frac{\lambda\beta
s}{\lambda_1+\beta} \biggr)^n\frac{1}{n!\varGamma(n)}\sum
_{j=1}^{k-1}\frac
{\varGamma(n+k-j)}{(k-j)!}ds+I_2
\nonumber
\\[3pt]
&=&
\frac{\lambda\beta\lambda_1^k e^{-\lambda s}}{ (\lambda
_1+\beta )^{1+k}}\sum_{n=1}^{\infty} \biggl(
\frac{\lambda\beta
s}{\lambda_1+\beta} \biggr)^n\frac{1}{n!}\sum
_{l=1}^{k-1}\frac
{(n+l-1)!}{l!(n-1)!}ds
\nonumber
+I_2
\\[3pt]
&=& \frac{\lambda\beta\lambda_1^k e^{-\lambda s}}{ (\lambda
_1+\beta )^{1+k}}\sum_{n=1}^{\infty}
\biggl(\frac{\lambda\beta
s}{\lambda_1+\beta} \biggr)^n\frac{1}{n!} \Biggl[\sum
_{l=0}^{k-1}C_{n+l-1}^l-1
\Biggr]ds+I_2
\nonumber
\\[3pt]
&=& \frac{\lambda\beta\lambda_1^k e^{-\lambda s}}{ (\lambda
_1+\beta )^{1+k}}\sum_{n=1}^{\infty}
\biggl(\frac{\lambda\beta
s}{\lambda_1+\beta} \biggr)^n\frac{1}{n!}
\bigl[C_{n+k-1}^{k-1}-1 \bigr]ds+I_2
\nonumber
\\[3pt]
&=& \frac{\lambda\beta}{\lambda_1+\beta}\sum
_{n=1}^{\infty
}e^{-\lambda s}\frac{ (\lambda s )^n}{n!}
\frac{\lambda_1^k\beta
^n }{ (\lambda_1+\beta )^{n+k}}\frac{\varGamma(n+k)}{n!\varGamma
(k)}ds
\nonumber
\\[3pt]
&& {}-\frac{\lambda\beta\lambda_1^k e^{-\lambda s}}{ (\lambda_1+\beta
 )^{1+k}}\sum_{n=1}^{\infty}
\biggl(\frac{\lambda\beta s}{\lambda
_1+\beta} \biggr)^n\frac{1}{n!}ds+I_2
\nonumber
\\[3pt]
&=&\frac{\lambda\beta}{\lambda_1+\beta}\sum_{n=1}^{\infty
}e^{-\lambda s}
\frac{ (\lambda s )^n}{n!}\frac{\lambda_1^k\beta
^n }{ (\lambda_1+\beta )^{n+k}}\frac{\varGamma(n+k)}{n!\varGamma
(k)}ds
\nonumber
\\[3pt]
&& {}-\frac{\lambda\beta\lambda_1^k e^{-\lambda s}}{ (\lambda_1+\beta
 )^{1+k}} \bigl[e^{-\lambda s\frac{\beta}{\lambda_{1}+\beta
}}-1 \bigr]ds
\nonumber
\\[3pt]
&&{}+\exp \biggl\{-\lambda s \biggl(1-\frac{\beta}{\lambda_{1}+\beta} \biggr) \biggr\}
\frac{\lambda\lambda_1^k\beta}{(\lambda_1+\beta
)^{k+1}}ds
\nonumber
\\[3pt]
&=&\frac{\lambda\beta}{\lambda_1+\beta}\sum_{n=0}^{\infty
}e^{-\lambda s}
\frac{ (\lambda s )^n}{n!}\frac{\lambda_1^k\beta
^n }{ (\lambda_1+\beta )^{n+k}}\frac{\varGamma(n+k)}{n!\varGamma
(k)}ds
\nonumber
\\[3pt]
&=& \frac{\lambda\beta}{\lambda_1+\beta}\sum_{n=0}^{\infty}P
\bigl\{ N_{\lambda}(s)=n \bigr\}P \bigl\{N_{\beta}
\bigl(E_{\lambda_1}(k) \bigr)=n \bigr\}ds.\label{PTkdsNEN1}
\qedhere
\end{eqnarray}
\end{proof}

By integrating the left part of the formula (\ref{PTkdsNEN1}) by $s$ we get:
\begin{eqnarray*}
P \bigl\{T_k^E<\infty \bigr\}&=& \frac{\lambda\beta}{\lambda_1+\beta
}\sum
_{n=0}^{\infty}\int_{0}^{\infty}e^{-\lambda s}
\frac{ (\lambda
s )^n}{n!} ds P \bigl\{N_{\beta} \bigl(G_{\lambda_1}(k)
\bigr)=n \bigr\}
\nonumber
\\
&=& \frac{ \beta}{\lambda_1+\beta}\sum_{n=0}^{\infty} P
\bigl\{N_{\beta
} \bigl(G_{\lambda_1}(k) \bigr)=n \bigr\} =
\frac{ \beta}{\lambda_1+\beta},
\end{eqnarray*}
which coincides with the formula \eqref{PTkNEN}.

\section{First passage time of the subordinated Poisson process $N_1(G_N(s))$}\label{sec4}

In this section we study first passage time for the process
$N_1(G_N(s))$ defined as
\[
\tilde{T}_k=inf \bigl\{s\geq0:N_1
\bigl(G_N(s) \bigr)\geq k \bigr\}.
\]
We have:
\begin{eqnarray*}
P \{\tilde{T}_k<s \}&=& P \bigl\{N_1
\bigl(G_N(s) \bigr)\geq k \bigr\}= \sum
_{j=k}^{\infty}\int_{0}^{\infty}e^{-\lambda_1 z}
\frac{ (\lambda
_1 z )^j}{j!} P \bigl\{G_N(s)\in dz \bigr\}
\nonumber
\\
&=& e^{-\lambda s}\sum_{j=k}^{\infty}
\biggl(\frac{\lambda_1}{\lambda
_1+\beta} \biggr)^j\frac{1}{j!}\sum
_{n=1}^{\infty} \biggl(\frac{\lambda s
\beta^{\alpha}}{ (\lambda_1+\beta )^{\alpha}}
\biggr)^n\frac
{\varGamma(\alpha n+j)}{n!\varGamma(\alpha n)},\label{PTkNGN2}
\end{eqnarray*}
and, therefore,
\begin{align*}
&P \{\tilde{T}_k\in ds \}/ds
\nonumber\\
&\quad =\frac{d}{ds} e^{-\lambda s}
\sum_{j=k}^{\infty} \biggl(\frac{\lambda_1}{\lambda_1+\beta}
\biggr)^j\frac
{1}{j!}\sum_{n=1}^{\infty}
\biggl(\frac{\lambda s \beta^{\alpha}}{
(\lambda_1+\beta )^{\alpha}} \biggr)^n\frac{\varGamma(\alpha
n+j)}{n!\varGamma(\alpha n)}
\nonumber
\\
&\quad = e^{-\lambda s}\sum_{j=k}^{\infty}
\biggl(\frac{\lambda
_1}{\lambda_1+\beta} \biggr)^j\frac{1}{j!}\sum
_{n=1}^{\infty
}s^{n-1}(n-\lambda s) \biggl(
\frac{\lambda s \beta^{\alpha}}{ (\lambda
_1+\beta )^{\alpha}} \biggr)^n\frac{\varGamma(\alpha n+j)}{n!\varGamma
(\alpha n)}.
\end{align*}
When $\alpha=1$, the process $G_{N}(t)$ becomes
$E_{N}(t)$, that is, the compound Poisson process with exponentially
distributed jumps. For this process we denote the hitting times by
\[
\tilde{T}_k^E=inf \bigl\{s\geq0:N_1
\bigl(E_N(s) \bigr)\geq k \bigr\}.
\]
We obtain:
\begin{eqnarray*}
P \bigl\{\tilde{T}_k^E<s \bigr\}&=& P \bigl
\{N_1 \bigl(G_N(s) \bigr)\geq k \bigr\}= \sum
_{j=k}^{\infty}\int_{0}^{\infty}e^{-\lambda_1 z}
\frac{ (\lambda
_1 z )^j}{j!} P \bigl\{E_N(s)\in dz \bigr\}
\nonumber
\\
&=& e^{-\lambda s}\sum_{j=k}^{\infty}
\frac{\lambda_1^j\lambda\beta
s}{ (\lambda_1+\beta )^{j+1}}\mathcal{E}_{1,2}^{j+1} \biggl(
\frac{\lambda\beta s}{\lambda_{1}+\beta} \biggr).\label{PTkNEN2}
\end{eqnarray*}
Therefore,
\begin{eqnarray*}
P \bigl\{\tilde{T}_k^E\in ds \bigr\}/ds&=&
\frac{d}{ds} e^{-\lambda
s}\sum_{j=k}^{\infty}
\frac{\lambda_1^j\lambda\beta s}{ (\lambda
_1+\beta )^{j+1}}\mathcal{E}_{1,2}^{j+1} \biggl(
\frac{\lambda\beta s}{\lambda_{1}+\beta} \biggr).
\end{eqnarray*}
Taking into account that
$\frac{d}{ds} (as\mathcal{E}_{1,2}^{\gamma} (as )
)=a\mathcal{E}_{1,1}^{\gamma} (as )$,
we obtain:
\begin{eqnarray*}
P \bigl\{\tilde{T}_k^E\in ds \bigr\}/ds&=&\sum
_{j=k}^{\infty}\frac
{\lambda_1^j}{ (\lambda_1+\beta )^{j}}\frac{d}{ds}e^{-\lambda
s}
\frac{\lambda\beta s}{\lambda_1+\beta} \mathcal{E}_{1,2}^{j+1} \biggl(
\frac{\lambda\beta s}{\lambda_{1}+\beta} \biggr)
\nonumber
\\
&=&\sum_{j=k}^{\infty}\frac{\lambda_1^j}{ (\lambda_1+\beta
)^{j}}e^{-\lambda s}
\biggl[ \biggl(-\lambda\frac{\lambda\beta s}{\lambda
_1+\beta} \mathcal{E}_{1,2}^{j+1}
\biggl(\frac{\lambda\beta s}{\lambda
_{1}+\beta} \biggr)
\\
&&{}+\frac{\lambda\beta}{\lambda_1+\beta} \mathcal {E}_{1,1}^{j+1}
\biggl(\frac{\lambda\beta s}{\lambda_{1}+\beta} \biggr) \biggr) \biggr]
\nonumber
\\
&=&e^{-\lambda s}\frac{\lambda\beta}{\lambda_1+\beta} \sum_{j=k}^{\infty}
\frac{\lambda_1^j}{ (\lambda_1+\beta )^{j}}
\\
&&{} \times \biggl[ \mathcal{E}_{1,1}^{j+1} \biggl(
\frac{\lambda\beta
s}{\lambda_{1}+\beta} \biggr)-\lambda s \mathcal{E}_{1,2}^{j+1}
\biggl(\frac{\lambda\beta s}{\lambda_{1}+\beta} \biggr) \biggr].
\end{eqnarray*}

We summarize the above reasonings in the following form.
%
\begin{lemma}
For the process $N_1(G_N(s))$ the law of $\tilde{T}_k$ is given by the
following formula:
\begin{eqnarray*}
P \{\tilde{T}_k\in ds \}/ds &=& e^{-\lambda s}\sum
_{j=k}^{\infty} \biggl(\frac{\lambda_1}{\lambda_1+\beta}
\biggr)^j\frac
{1}{j!}\sum_{n=1}^{\infty}s^{n-1}(n-
\lambda s)
\\
&&{} \times \biggl(\frac{\lambda s \beta^{\alpha}}{ (\lambda_1+\beta
 )^{\alpha}} \biggr)^n\frac{\varGamma(\alpha n+j)}{n!\varGamma(\alpha
n)}\label{PTkdsNGN2}.
\end{eqnarray*}
In the case when $\alpha=1$, that is, for the process $N_1(E_N(s))$, we have:
\begin{eqnarray*}
P \bigl\{\tilde{T}_k^E\in ds \bigr\}/ds&=&e^{-\lambda s}
\frac{\lambda
\beta}{\lambda_1+\beta} \sum_{j=k}^{\infty}
\frac{\lambda_1^j}{
(\lambda_1+\beta )^{j}}
\\
&&{} \times \biggl[ \mathcal{E}_{1,1}^{j+1} \biggl(
\frac{\lambda\beta
s}{\lambda_{1}+\beta} \biggr)-\lambda s \mathcal{E}_{1,2}^{j+1}
\biggl(\frac{\lambda\beta s}{\lambda_{1}+\beta} \biggr) \biggr].
\end{eqnarray*}
\end{lemma}

\section{Gamma process subordinated to the Poisson process with a drift}\label{sec5}

Let $N(t)$ be the Poisson process with the intensity parameter $\lambda$.
Consider the process with a drift
%
%
\begin{equation}
N(t)+at , \quad  a>0\label{Na}.
\end{equation}

The probability law of the process $N(t)+at$ can be written in the
following form (see, e.g., \cite{B-D'O}, Theorem 1):
%
%
\begin{equation}
p_x(t)=e^{-\lambda t}\sum_{k=0}^{\infty}
\frac{(\lambda
t)^k}{k!}\delta(x-k-at),\quad  x\geq a t, a>0, t>0.\label{px}
\end{equation}
The Laplace transform of the process \eqref{Na} is:
%
%
\begin{equation}
\E e^{-u(at+N(t))}= e^{-t(a u+\lambda(1-e^{-u}))}, \label{LT1}
\end{equation}
and, correspondingly, the Laplace exponent is:
%
%
\begin{equation}
f_{N+a}(u)= a u+\lambda \bigl(1-e^{-u} \bigr).\label{fNa}
\end{equation}

Let $G(t)$ be the Gamma process with parameters $(\alpha, \beta)$, that
is, with the Laplace exponent $f(u)=\alpha\log ( 1+\frac{u}{\beta
} )$ . Denote the probability density of $G(t)$ by $h_{G(t)}(y)$.

For $a\geq0$ consider the process
%
%
\begin{equation}
G_{N+a}(t)=G\bigl(at+N(t)\bigr), \quad  t\ge0.\label{Y}
\end{equation}
Its Laplace transform has the following form:
%
%
\begin{equation}
\E e^{-u G(at+N(t))}=e^{-t  (a\alpha\log (1+\frac{u}{\beta
} )+\lambda (1- (1+\frac{u}{\beta} )^{-\alpha}
) )}, \label{LT2}
\end{equation}
therefore, the Laplace exponent and the corresponding Levy measure are
given by the following formulas:
%
%
\begin{equation}
f_{G_{N+a}}(u)=a\alpha\log \biggl(1+\frac{u}{\beta} \biggr)+\lambda
\biggl(1- \biggl(1+\frac{u}{\beta} \biggr)^{-\alpha}
\biggr),\label{fGNa}
\end{equation}
and
%
%
\begin{equation}
\nu(du)=e^{-\beta u}u^{-1} \biggl(a\alpha+\frac{\lambda\beta^{\alpha
}}{\varGamma(\alpha)}u^{\alpha}
\biggr)du, \quad \lambda>0, \alpha>0, \beta>0, \label{du}
\end{equation}
and we note that the process $G(at+N(t))$ coincides in distribution
with the sum of independent processes $\tilde{G}_N(t)+\tilde{G}(at)$,
where $\tilde{G}_N(t)$ is the compound Poisson-Gamma process and $\tilde
{G}(t)$ is the Gamma process.

The distribution of the process $G(at+N(t))$ can be calculated as follows:
\begin{eqnarray*}
P \bigl\{G_{N+a}(t)\in dy \bigr\}&=& \int_{0}^{\infty}
h_{G(s)}(y)p_s(t)ds
\\
&=&\int_{0}^{\infty}\frac{\beta^{\alpha s}}{\varGamma
(\alpha s)}y^{\alpha s-1}e^{-\beta y}e^{-\lambda t}
\sum_{k=0}^{\infty
}\frac{(\lambda t)^k}{k!}
\delta(s-k-at)ds dy
\\
&=&e^{- y \beta- \lambda t}\sum_{k=0}^{\infty}
\frac{(\lambda
t)^k}{k!}\frac{1}{y}\int_{0}^{\infty}
\frac{\beta^{\alpha s}}{\varGamma
(\alpha s)}y^{\alpha s} \delta(s-k-at)ds dy
\\
&=& e^{- y\beta-\lambda t}y^{\alpha a t-1} \beta^{\alpha a t}\sum
_{k=0}^{\infty} \frac{ (\lambda t (\beta y )^{\alpha} )^k}{k!
\varGamma(\alpha(k+at))}
\\
&=&e^{- y\beta-\lambda t}\frac{(y\beta)^{\alpha a t}}{y}\varPhi\bigl(\alpha, \alpha a t, \lambda t
(\beta y)^{\alpha}\bigr)dy,\quad  a\neq0,
\end{eqnarray*}
where in the last line the Wright
function appears:
%
%
\begin{equation}
\varPhi(\rho,\delta,z)=\sum_{k=0}^{\infty}
\frac{z^{k}}{k!\varGamma(\rho
k+\delta)},\quad z\in C,\text{ }\rho\in(-1,0)\cup(0,\infty ),\text{
}%
\delta\in C; \label{Wr}
\end{equation}
We summarize the above reasonings in the following lemma.

%
\begin{lemma}
The process $G_{N+a}(t)=G(at+N(t))$ is the L\'evy process with the
Bern\v{s}tein function and L\'evy measure given by \eqref{fGNa}, \eqref
{du}, and its probability distribution is given by
\begin{eqnarray*}
P \bigl\{G_{N+a}(t)\in dy \bigr\}&=& e^{- y\beta-\lambda t}
\frac{(y\beta
)^{\alpha a t}}{y}\varPhi\bigl(\alpha, \alpha a t, \lambda t (\beta
y)^{\alpha
}\bigr)dy,\quad a\neq0.\label{PGNady}
\end{eqnarray*}
\end{lemma}

%
\begin{remark} The first two moments of the processes $at+N(t)$ and
$G(at+N(t))$ are:
\begin{equation*}
\mathsf{E}\bigl(at+N(t)\bigr)= (\lambda+a)t; \qquad Var\bigl(at+N(t)\bigr)=
\lambda t;
\end{equation*}
\begin{equation*}
cov\bigl(at+N(t),as+N(s)\bigr)= \lambda\min(t,s).
\end{equation*}
\begin{equation*}
\mathsf{E}G\bigl(at+N(t)\bigr)= \alpha\beta^{-1}(\lambda+a)t;
\end{equation*}
\begin{equation*}
VarG\bigl(at+N(t)\bigr)= \alpha\beta^{-2}(\lambda\alpha+\lambda+a)t;
\end{equation*}
\begin{equation*}
cov\bigl(G\bigl(at+N(t)\bigr),G\bigl(as+N(s)\bigr)\bigr)= \alpha
\beta^{-2}(\lambda\alpha+\lambda +a)\min(t,s).
\end{equation*}
\end{remark}

When the parameter $\alpha=1$, that is, the process $G_{N+a}(t)$
becomes $E_{N+a}(t)$, we obtain:
%
%
\begin{eqnarray}
P \bigl\{E_{N+a}(t)\in dy \bigr\} &=&e^{- y\beta-\lambda t}
\beta^{\frac{at+1}{2}} \biggl(\frac{y}{\lambda
t} \biggr)^{\frac{at-1}{2}}
I_{at-1} (2 \sqrt{\lambda\beta t y} ) \label{PENady}
\end{eqnarray}
since
\begin{equation*}
\varPhi(1, a t, z)=z^{\frac{1-at}{2}}I_{at-1} (2\sqrt{z} ),
\end{equation*}
where $I_{k}$ is the modified Bessel function of the first kind (see,
e.g., \cite{Sn}):
\begin{equation*}
I_{k}(z)=\sum_{n=0}^{\infty}
\frac{ ( z/2 ) ^{2n+k}}{n!(n+k)!}. \label{Bessel}
\end{equation*}
Note that the distribution \eqref{PENady} was presented in \cite
{Sato,W,W2001} for the case when $\beta=1$.

%
\begin{remark} We recall that for the case when the shift $a=0$ the
distributions of $G_N(t)$ and $E_N(t)$ have the atom at zero and are
given by the following formulas:
\begin{align*}
P \bigl\{ G_{N}(t)\in ds \bigr\} &=e^{-\lambda t}
\delta_{\{0\}}(ds)+e^{-\lambda t-\beta s}\frac
{1}{s}\varPhi \bigl(
\alpha,0,\lambda t(\beta s)^{\alpha} \bigr) ds,
\\
P \bigl\{ E_{N}(t)\in ds \bigr\} &=e^{-\lambda t}
\delta_{\{0\}}(ds)+e^{-\lambda t-\beta s}\frac{\sqrt{%
\lambda tbs}}{s}I_{1} (
2\sqrt{\lambda t\beta s} )ds .
\end{align*}
\end{remark}

Let $N_1(t)$ be the Poisson process with intensity parameter $\lambda
_1$. Consider the time-changed process
%
%
\begin{equation}
\label{N1Gn+a} N_1\bigl(G_{N+a}(t)\bigr)=N_1
\bigl(G\bigl(N(t)+at\bigr)\bigr).
\end{equation}

%
\begin{theorem}
Probability mass function of the process $N_1(G_{N+a}(t))$ is given by
%
%
\begin{eqnarray}
p_{k}(t)&=&P \bigl\{N_1\bigl(G_{N+a}(t)
\bigr)=k \bigr\}
\nonumber
\\
&=&e^{-\lambda t}\frac{ \lambda_1^k}{k!}\sum_{n=0}^{\infty}
\frac
{(\lambda t )^n}{n!}\frac{ \beta^{\alpha(n+at)}}{\varGamma(\alpha
(n+at))}\frac{\varGamma(\alpha(n+at)+k)}{ (\lambda_1+\beta
)^{\alpha(n+at)+k}}.\label{PN1GNa}
\end{eqnarray}
The probabilities $p_k(t)$ satisfy the following system of
difference-differential equations:
%
%
\begin{align}
\label{dtpGat} \frac{d}{dt}p_k(t)&= - \biggl(a \alpha\log
\biggl(1+\frac{\lambda_1}{\beta
} \biggr)+\lambda \biggl(1- \biggl(\frac{\beta}{\lambda_1+\beta}
\biggr)^{\alpha} \biggr) \biggr)p_k(t)
\nonumber
\\
&\quad +\sum_{m=1}^{k}\frac{1}{m!}
\biggl(\frac{\lambda_1}{\lambda_1+\beta
} \biggr)^m \biggl(a \alpha\varGamma(m)+
\lambda \biggl(\frac{\beta}{\lambda
_1+\beta} \biggr)^{\alpha}\frac{\varGamma(m+\alpha)}{\varGamma(\alpha)}
\biggr)p_{k-m}(t).
\end{align}
\end{theorem}

\begin{proof}
The probability mass function of the process $N_1(G_{N+a}(t))$ can be
obtained by standard conditioning arguments (see, e.g., the general
result for subordinated L\'{e}vy processes in \cite{Sato}, Theorem 30.1).\vadjust{\goodbreak}
\begin{eqnarray*}
p_{k}(t)&=&P \bigl\{N_1\bigl(G\bigl(at+ N(t)\bigr)
\bigr)=k \bigr\}
\\
&=& \int_{0}^{\infty}P \bigl\{N_1(s)=k
\bigr\}P \bigl\{G\bigl(at+ N(t)\bigr)\in ds \bigr\}
\nonumber
\\
&=&\int_{0}^{\infty} e^{-\lambda_1 s}
\frac{(\lambda_1 s)^k}{k!} e^{-\beta s -\lambda t} \sum_{n=0}^{\infty}
\frac{(\lambda t)^n}{n!} \frac{1}{s} \frac{\beta^{\alpha(n+at)}}{\varGamma(n+at)} s^{\alpha
(n+at)}
\nonumber
\\
&=&e^{-\lambda t}\sum_{n=0}^{\infty}
\frac{(\lambda t )^n}{n!}\frac{
\lambda_1^k}{k!}\frac{ \beta^{\alpha(n+at)}}{\varGamma(\alpha(n+at))}\frac
{\varGamma(\alpha(n+at)+k)}{ (\lambda_1+\beta )^{\alpha(n+at)+k}}.
\end{eqnarray*}
Using the formula \eqref{dtpkf}, we obtain the equations \eqref{dtpGat}.
\end{proof}

%
\begin{theorem}
Probability mass function of the process $N_1(E_{N+a}(t))$ is given by
%
%
\begin{eqnarray}
p_{k}^E(t)&=&P \bigl\{N_1
\bigl(E_{N+a}(t)\bigr)=k \bigr\}
\nonumber
\\
&=&e^{-\lambda t}\frac{\varGamma(k+at)}{k!} \biggl(\frac{ \lambda
_1}{\lambda_1+\beta}
\biggr)^k \biggl(\frac{ \beta}{\lambda_1+\beta} \biggr)^{at}
\mathcal{E}_{1,at}^{k+at} \biggl(\frac{\lambda\beta t}{\lambda
_1+\beta}
\biggr).\label{PN1ENa}
\end{eqnarray}
The probabilities $p_k^E(t)$ satisfy the following system of
difference-differential equations:
\begin{eqnarray*}
\frac{d}{dt}p_k^E(t)&=&- \biggl(a \log \biggl(1+
\frac{\lambda_1}{\beta
} \biggr)+\frac{\lambda\lambda_1}{\lambda_1+\beta} \biggr)p_k^E(t)
+\sum_{m=1}^{k} \biggl(\frac{\lambda_1}{\lambda_1+\beta}
\biggr)^m
\\
&&{} \times \biggl(\frac{a}{m} +\frac{\lambda\beta}{\lambda_1+\beta} \biggr)p_{k-m}^E(t).
\end{eqnarray*}
\end{theorem}

%
\begin{remark}
Distribution \eqref{PN1GNa} can be also obtained from the probability
generating function.
Using the general formula \eqref{Gf} for the process \eqref{N1Gn+a} we find:
\begin{eqnarray*}
G^F(u,t)&=&e^{-t f(\lambda_1(1-u))}=e^{-\lambda t}\sum
_{n=0}^{\infty} \frac{(\lambda t)^n}{n!} \biggl(1+
\frac{\lambda_1(1-u)}{\beta} \biggr)^{-\alpha(n+at)}
\nonumber
\\[3pt]
&=&e^{-\lambda t}\sum_{n=0}^{\infty}
\frac{(\lambda t)^n}{n!} \biggl(1+\frac{\lambda_1}{\beta} \biggr)^{-\alpha(n+at)} \biggl(1-
\frac{\lambda
_1 u}{\lambda_1+\beta} \biggr)^{-\alpha(n+at)}
\nonumber
\\[3pt]
&=&e^{-\lambda t} \biggl(\frac{\beta}{\lambda_1+\beta} \biggr)^{\alpha
at}\sum
_{n=0}^{\infty} \frac{(\lambda t)^n}{n!} \biggl(1+
\frac{\lambda
_1}{\beta} \biggr)^{-\alpha n} \sum_{k=0}^{\infty}
\frac{1}{k!} \biggl(\frac{\lambda_1 u}{\lambda_1+\beta} \biggr)^{k}
\\[3pt]
&&{} \times\frac{\varGamma(\alpha(n+at)+k)}{\varGamma(\alpha(n+at))}
\nonumber
\\[3pt]
&=&\sum_{k=0}^{\infty} u^k
\Biggl[e^{-\lambda t} \biggl(\frac{\beta
}{\lambda_1+\beta} \biggr)^{\alpha at}
\frac{1}{k!} \biggl(\frac{\lambda_1
}{\lambda_1+\beta} \biggr)^{k}\sum
_{n=0}^{\infty} \frac{1}{n!} \biggl(
\frac
{\lambda t \beta^{\alpha}}{(\lambda_1+\beta)^{\alpha}} \biggr)^n
\\[3pt]
&&{} \times\frac{\varGamma(\alpha(n+at)+k)}{\varGamma(\alpha
(n+at))} \Biggr].
\end{eqnarray*}
Therefore,
\begin{eqnarray*}
p_k(t)= e^{-\lambda t} \biggl(\frac{\beta}{\lambda_1+\beta}
\biggr)^{\alpha
at}\frac{1}{k!} \biggl(\frac{\lambda_1}{\lambda_1+\beta}
\biggr)^{k}\sum_{n=0}^{\infty}
\frac{1}{n!} \biggl(\frac{\lambda t \beta^{\alpha
}}{(\lambda_1+\beta)^{\alpha}} \biggr)^n
\frac{\varGamma(\alpha
(n+at)+k)}{\varGamma(\alpha(n+at))},\label{PiG}
\end{eqnarray*}
which coincides with \eqref{PN1GNa}.
\end{remark}

%
\begin{remark} The first two moments of the process $N_{1}(G(at+N(t)))$ are:
\begin{equation*}
\mathsf{E}N_{1}\bigl(G_{a+N}(t)\bigr)=
\lambda_1\alpha\beta^{-1}(\lambda+a)t;
\end{equation*}
\begin{equation*}
VarN_{1}\bigl(G_{a+N}(t)\bigr)= \lambda_1
\alpha\beta^{-2}\bigl(\lambda_1(\lambda \alpha+
\lambda+a)+(\lambda+a)\beta\bigr)t;
\end{equation*}
\begin{equation*}
cov\bigl(N_{1}\bigl(G_{a+N}(t)\bigr),N_{1}
\bigl(G_{a+N}(s)\bigr)\bigr)= \lambda_1\alpha\beta
^{-2}\bigl(\lambda_1(\lambda\alpha+\lambda+a)+(\lambda+a)
\beta\bigr)\min(t,s).
\end{equation*}
\end{remark}

Figures \ref{figLeft} and \ref{figRight}
show the behavior of the probabilities (\ref{PN1GNa}) and (\ref
{PN1ENa}), for various choices of $t$ $ (t=1,2,3)$.

%
\begin{figure}
\includegraphics{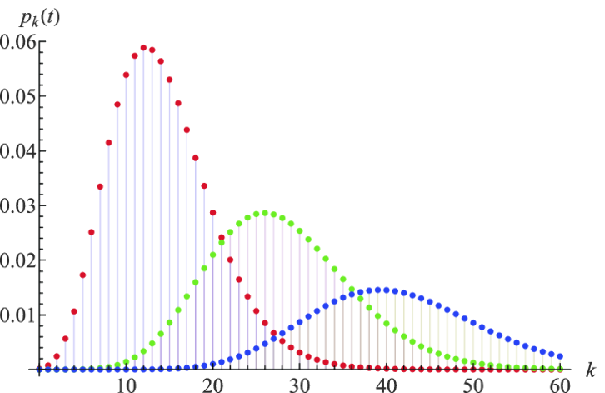}
\caption{Probabilities (\ref{PN1GNa}), for values of $a=5$, $\alpha=2$,
$\beta=0.8$, $ \lambda=1$, $\lambda_1=1$}
\label{figLeft}
\end{figure}

%
\begin{figure}
\includegraphics{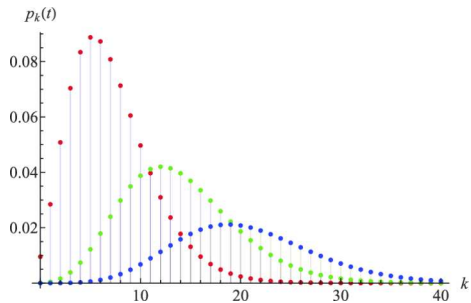}
\caption{Probabilities (\ref{PN1ENa}), for values of $a=5$, $\beta=0.8,
\lambda=1, \lambda_1=1$}
\label{figRight}
\end{figure}

\section{Iterated Bessel transforms}\label{sec6}

Consider the L\'evy process $Z_G(t)=G(at+N(X(t)))$, $t\ge0$, $a\geq0$,
where we assume that ${G(t)}$, ${N(t)}$ and $X(t)$ are independent L\'
evy processes, ${G(t)}$ is the Gamma process with parameters $(\alpha,
\beta)$, ${N(t)}$ is the Poisson process with parameter $\lambda$, and
$\nu_X(du)$ is the L\'evy measure of $X(t)$.

Using Theorem 30.1 from \cite{Sato}, we can calculate the L\'evy
measure of $Z_G(t)$. We obtain:
%
%
\begin{equation}
\label{nuZG} \nu_{Z_G}(dx)= e^{-\beta x} \Biggl(a\alpha
x^{-1}+\int_{0}^{\infty} e^{-\lambda u}
x^{-1}\varPhi\bigl(\alpha,0 ,\lambda u(\beta x)^{\alpha}\bigr)
\nu_X (du) \Biggr)dx,
\end{equation}
where $\varPhi(\rho,0 ,z)$ is the Wright function.

In the case when $\alpha=1$, that is, when the process $G(t)$ becomes
$E(t)$, the exponential process, we obtain the process
$Z_E(t)=E(at+N(X(t)))$ with the L\'evy measure given by the formula
%
%
\begin{equation}
\label{nuZE} \nu_{Z_E}(dx)= e^{-\beta x} \Biggl(ax^{-1}+
\int_{0}^{\infty} e^{-\lambda u}\sqrt{\lambda u \beta
x^{-1}}\,I_1 (2\sqrt{\lambda u \beta x}\, )
\nu_X (du) \Biggr)dx,
\end{equation}
since $\varPhi(1 ,0 ,z)=\sqrt{z}I_1(2\sqrt{z})$.

The process $Z_E(t)$ with $\beta=1$, $\lambda=1$ was considered in \cite
{Sato,W,W2001}.

In what follows we will consider the process $G(at+N(t))$ for $\alpha=1$.

Define the following iteration of the processes $E(at+N(t))$:
%
%
\begin{eqnarray}
\label{Xn} &X_0(t)=t,
\nonumber
\\
&X_1(t)=E_1\bigl(a_1t+N_1
\bigl(X_0(t)\bigr)\bigr),
\nonumber
\\
& \ldots
\\
&X_n(t)=E_n\bigl(a_nt+N_n
\bigl(X_{n-1}(t)\bigr)\bigr),
\nonumber
\end{eqnarray}
where $E_i(t), i=1,\ldots,n$, are independent exponential processes
with parameters $\beta_i, i=1,\ldots,n$, $N_i, i=1,\ldots,n$, are
independent Poisson processes with intensity parameters $\lambda_i,
i=1,\ldots,n$. For the process $X_n(t)$, we are able to calculate in
closed form its L\'evy measure $\nu_n(du)$ and the corresponding Bern\v
{s}tein function $f_n(u)$. This result is presented in the next theorem.

%
\begin{theorem}\label{thm7}
Let $X_n(t)$ be the process defined by the iteration formulas \eqref{Xn}.
Then the following holds:

\noindent(i) if $\beta_i=\lambda_i=1, i=1,\xch{\ldots}{..},n$, then
%
%
\begin{gather}
\label{nun1} \nu_n(du)=\Biggl(u^{-1}\sum
_{k=0}^{n-1}e^{-u\frac
{1}{1+k}}(a_{n-k}-a_{n-k-1})+e^{-u\frac{1}{n}}
\frac{1}{n^2}\Biggr)du,
\\
%
%
\label{fn1} f_n(u)=\sum_{k=0}^{n-1}(a_{n-k}-a_{n-k-1})
\log \bigl(1+u(1+k) \bigr)+\frac{u}{1+nu};
\end{gather}

\noindent(ii) if $\beta_i=\beta\neq1,\lambda_i=\lambda\neq1, i=1,\xch
{\ldots}{..},n$, then
%
%
\begin{equation}
\label{nun3} \nu_n(du)=\Biggl(u^{-1}\sum
_{k=0}^{n-1}e^{-u\frac{\beta^{k+1}}{\gamma(\lambda
,\beta,k+1)}}(a_{n-k}-a_{n-k-1})+e^{-u\frac{\beta^n}{\gamma(\lambda
,\beta,n)}}
\frac{(\lambda\beta)^n}{(\gamma(\lambda,\beta,n))^2}\Biggr)du,
\end{equation}
%
%
\begin{equation}
\label{fn3} f_n(u)=\sum_{k=0}^{n-1}(a_{n-k}-a_{n-k-1})
\log \biggl(1+u\frac{\gamma
(\lambda,\beta,k+1)}{\beta^{k+1}} \biggr)+\lambda^n
\frac{u}{\beta
^n+u\gamma(\lambda,\beta,n)},
\end{equation}
where $\gamma(\lambda,\beta,m)=\sum_{j=1}^{m}\lambda^{m-j}\beta
^{j-1}=(\lambda^m-\beta^m)(\lambda-\beta)^{-1}$, $a_0=0$.
\end{theorem}
\begin{proof}
We present the proof for the case $\beta_i=\beta\neq1,\lambda_i=\lambda
\neq1, i=1,\ldots,n$. We prove the claimed results by induction.

For $n=1$ the formula (\ref{nun3}) holds. Suppose that the result is
true for $n=m$ ($m\geq1$), that is,
\begin{equation*}
\nu_m(du)=\Biggl(u^{-1}\sum_{k=0}^{m-1}e^{-u\frac{\beta
^{k+1}}{\gamma(\lambda,\beta,k+1)}}(a_{m-k}-a_{m-k-1})+e^{-u\frac{\beta
^m}{\gamma(\lambda,\beta,m)}}
\frac{(\lambda\beta)mn}{(\gamma(\lambda
,\beta,m))^2}\Biggr)du.
\end{equation*}
We need to show that (\ref{nun3}) holds for $n=m+1$. We calculate $\nu
_{m+1}(dx)$:
\begin{align*}
\nu_{m+1}(dx)&=e^{-\beta x} \Biggl(a_{m+1}x^{-1}+
\int_{0}^{\infty
}e^{-\lambda u}\sqrt{\lambda u \beta
x^{-1}}I_1 (2\sqrt{\lambda u \beta x} )
\nu_m(du) \Biggr)dx
\nonumber
\\
&=e^{-\beta x} \Biggl(a_{m+1}x^{-1}+\sum
_{n=0}^{\infty} \frac{ (\lambda\beta )^{n+1}x^n}{n!(n+1)!} \int
_{0}^{\infty}e^{-\lambda u}u^{n+1}
\nu_m(du) \Biggr)dx
\nonumber
\\
&=e^{-\beta x}\Biggl(a_{m+1}x^{-1}+\sum
_{n=0}^{\infty} \frac{ (\lambda\beta )^{n+1}x^n}{n!(n+1)!} \int
_{0}^{\infty
}e^{-\lambda u}u^{n+1}
\nonumber
\\
&\quad {} \times \Biggl(u^{-1}\sum_{k=0}^{m-1}e^{-u\frac{\beta^{k+1}}{\gamma
(\lambda,\beta,k+1)}}(a_{m-k}-a_{m-k-1})
\nonumber\\
&\quad
+e^{-u\frac{\beta^m}{\gamma(\lambda,\beta,m)}}\frac{(\lambda\beta
)^m}{\gamma(\lambda,\beta,m)^2}\Biggr)du \Biggr)dx
\nonumber
\\
&=e^{-\beta x}\Biggl(a_{m+1}x^{-1}
\nonumber
\\
&\quad {}+\!\sum_{n=0}^{\infty}\!
\frac{ (\lambda\beta )^{n+1}x^n}{(n+1)!} \!\sum_{k=0}^{m-1}\!(a_{m-k}-a_{m-k-1})
\biggl(\!\lambda+\frac{\beta^{k+1}}{\gamma
(\lambda,\beta,k+1)}\! \biggr)^{\!-(n+1)}
\nonumber
\\
&\quad {}+\sum_{n=0}^{\infty}
\frac{ (\lambda\beta )^{n+1}x^n}{n!}\frac{ (\lambda
\beta )^m}{ (\gamma(\lambda,\beta,m) )^2} \biggl(\lambda +\frac{\beta^{m}}{\gamma(\lambda,\beta,m)}
\biggr)^{-(n+2)} \Biggr)dx
\\
&=e^{-\beta x}\Biggl(a_{m+1}x^{-1}
\\
&\quad {}+x^{-1}\sum_{k=0}^{m-1}(a_{m-k}-a_{m-k-1})
\bigl(e^{ x
\frac{\lambda\beta\gamma(k+1)}{\lambda\gamma(\lambda,\beta,k+1)+\beta
^{k+1}}}-1 \bigr)
\nonumber
\\
&\quad {}+\frac{ (\lambda\beta )^{m+1}}{ (\lambda
\gamma(\lambda,\beta,m)+\beta^{m} )^2}e^{\lambda\beta x \frac
{\gamma(\lambda,\beta,m)}{\lambda\gamma(\lambda,\beta,m)+\beta
^{m}}}\Biggr)dx
\nonumber
\\
&=\Biggl(e^{-\beta x}x^{-1} \Biggl(a_{m+1}-
\sum_{k=0}^{m-1}(a_{m-k}-a_{m-k-1})
\Biggr)
\\
&\quad {}+e^{-x \frac{\beta^{k+2}}{\gamma(\lambda,\beta,k+2)}}x^{-1}\sum_{k=0}^{m-1}(a_{m-k}-a_{m-k-1})
\nonumber
\\
&\quad {}+e^{-x \frac{\beta^{m+1}}{\gamma(\lambda,\beta,m+1)}} \frac{
(\lambda\beta )^{m+1}}{ (\gamma(\lambda,\beta,m+1)
)^2}\Biggr)dx
\nonumber
\\
&=\Biggl(x^{-1}\sum_{k=0}^{m}e^{-x\frac{\beta^{k+1}}{\gamma(\lambda,\beta
,k+1)}}(a_{m+1-k}-a_{m-k})
\nonumber\\
&\quad +e^{-x \frac{\beta^{m+1}}{\gamma(\lambda,\beta
,m+1)}}
\frac{ (\lambda\beta )^{m+1}}{ (\gamma(\lambda
,\beta,m+1) )^2}\Biggr)dx
\end{align*}
Therefore, the formula (\ref{nun3}) is true.
\end{proof}

%
\begin{remark}
If $X_0=\lambda t$ in (\ref{Xn}) and $\beta_i=\lambda_i=1$ or $\beta
_i=1, \lambda_1=\lambda, \lambda_i=1, i=2,3,\ldots,n$, then
\begin{equation*}
\label{nun31} \nu_n(du)=\Biggl(u^{-1}\sum
_{k=0}^{n-1}e^{-u\frac
{1}{1+k}}(a_{n-k}-a_{n-k-1})+e^{-u\frac{1}{n}}
\frac{\lambda}{n^2}\Biggr)du,
\end{equation*}
\vspace*{-6pt}
\begin{equation*}
\label{fn31} f_n(u)=\sum_{k=0}^{n-1}(a_{n-k}-a_{n-k-1})
\log \bigl(1+u(1+k) \bigr)+\lambda\frac{u}{1+nu};
\end{equation*}
\end{remark}

%
\begin{remark}\label{rem10}
Formulas (\ref{nun1})--(\ref{fn3}) become significantly simpler in the
case when $a_1=a_2=\cdots=a_n=a$, that is, when the shift is the same
at each step. We obtain:

(i) if $\beta_i=\lambda_i=1, i=1,\xch{\ldots}{..},n$, then
%
%
\begin{equation}
\label{nun21} \nu_n(du)=u^{-1}e^{-u\frac{1}{n}} \biggl(a+
\frac{1}{n^2}u \biggr)du,
\end{equation}
%
%
\begin{equation}
\label{fn21} f_n(u)=a\log (1+nu )+\frac{u}{1+nu};
\end{equation}

(ii) if $\beta_i=\beta\neq1,\lambda_i=\lambda\neq1, i=1,\xch{\ldots
}{..},n$, then
%
%
\begin{equation}
\label{nun23} \nu_n(du)=u^{-1}e^{-u\frac{\beta^n}{\gamma(\lambda,\beta,n)}} \biggl(a+
\frac{(\lambda\beta)^n}{(\gamma(\lambda,\beta,n))^2}u \biggr)du,
\end{equation}
%
%
\begin{equation}
\label{fn23} f_n(u)=a\log \biggl(1+u\frac{\gamma(\lambda,\beta,n)}{\beta^{n}} \biggr)+
\lambda^n\frac{u}{\beta^n+u\gamma(\lambda,\beta,n)},
\end{equation}
where $\gamma(\lambda,\beta,n)=\sum_{j=1}^{n}\lambda^{n-j}\beta
^{j-1}=(\lambda^n-\beta^n)(\lambda-\beta)^{-1}$.

We can conclude that the process $X_n(t)$, which is given by the
formula (\ref{Xn}) as some kind of $n$-th iteration of processes
$E(N(t)+at)$, under the assumption that $E(t)$ and $N(t)$ have the
parameters $\beta=\lambda=1$, coincides in distribution with the
process $E_{1/n}(N_{1/n}(t)+at)$:
\[
X_n(t)\stackrel{d} {=}E_{1/n}\bigl(N_{1/n}(t)+at
\bigr), %
\]
where the exponential process $E_{1/n}$ and the Poisson process
$N_{1/n}(t)$ have the parameters $\beta=\lambda=1/n$, and, therefore,
the distribution of $X_n(t)$ in such a case is given by the formula
\eqref{PENady} with $\beta=\lambda=1/n$.

Correspondingly, in the case (ii) the process $X_n(t)$ coincides in
distribution with the process $\tilde{E}(\tilde{N}(t)+at)$, where the
exponential process $\tilde{E}(t)$ has parameter $\frac{\beta^n}{\gamma
(\lambda,\beta,n)}$, and the Poisson process $\tilde{N}(t)$ has
parameter $\frac{\lambda^n}{\gamma(\lambda,\beta,n)}$.
\end{remark}

%
\begin{lemma}\label{nu}
Assume the process $X_1(t)=E (at+N(t) )$, where $E(t)$ is an
exponential process with parameter $\gamma$, $N(t)$ is the Poisson
process with parameter $\gamma$, that is, $X_1(t)$ has the L\'evy measure
\[
\nu(du)=e^{-u \gamma} \bigl(u^{-1}a+\gamma^2 \bigr)du.
\]
Then the process
\[
X_2(t)=\tilde{E} \bigl(at+\tilde{N}\bigl(X_1(t)\bigr)
\bigr),
\]
where $\tilde{E}(t)$ is the exponential process with parameter $\alpha
$, $\tilde{N}(t)$ is the Poisson process with parameter $\alpha$, has
the L\'evy measure of the form:
%
%
\begin{equation}
\nu_2(dx)=e^{-x\frac{\alpha\gamma}{\alpha+\gamma}} \biggl(ax^{-1}+ \biggl(
\frac{\alpha\gamma}{\alpha+\gamma} \biggr)^2 \biggr)dx.
\end{equation}
\end{lemma}

\begin{proof}
Using the formula (\ref{nuZE}) with $\lambda= \beta= \gamma$ we obtain:
\begin{align*}
\nu_2(dx)&= e^{-\alpha x}ax^{-1}dx
\\
&\quad {}+e^{-\alpha x}\int_{0}^{\infty}
e^{\alpha u}\sqrt{\alpha^2 u x^{-1}}I_1
\bigl(2\sqrt{\alpha^2 u x} \bigr)e^{-u \gamma}
\bigl(u^{-1}a+\gamma^2 \bigr)du dx
\\
&=e^{-\alpha x}ax^{-1}dx
\\
&\quad {}+e^{-\alpha x}\sum_{j=0}^{\infty}
\frac{(\alpha^2 x)^j\alpha^2}{j!
(j+1)!}\int_{0}^{\infty}
e^{- u(\alpha+\gamma)}u^{j+1} \bigl(u^{-1}a+\gamma^2
\bigr)dudx
\\
&=e^{-\alpha x} \Biggl(ax^{-1}+\sum_{j=0}^{\infty}
\frac{(\alpha^2
x)^j\alpha^2}{j! (j+1)!} \biggl(a\frac{j!}{(\alpha+\gamma)^{j+1}}+\gamma ^2
\frac{(j+1)!}{(\alpha+\gamma)^{{j+2}}} \biggr) \Biggr)dx
\\
&=e^{-x\frac{\alpha\gamma}{\alpha+\gamma}} \biggl(ax^{-1}+ \biggl(\frac
{\gamma\alpha}{\alpha+\gamma}
\biggr)^2 \biggr)dx.
\qedhere
\end{align*}
\end{proof}

Generalizing the first statement of Remark~\ref{rem10}, we obtain the next
interesting result. Let the iterated process be constructed according
to the formula
(\ref{nuZE}) with the following parameters at each step:
%
%
\begin{eqnarray}
\label{Xn2} &X_0(t)=t,
\nonumber
\\
&X_1(t)=E_1\bigl(a t+N_1
\bigl(X_0(t)\bigr)\bigr), \quad \beta_1=
\lambda_1=\frac{1}{c_1}
\nonumber
\\
& \ldots
\\
&X_n(t)=E_n\bigl(a t+N_n
\bigl(X_{n-1}(t)\bigr)\bigr), \quad \beta_n=\lambda_n=
\frac
{1}{c_n}.
\nonumber
\end{eqnarray}
%
%
\begin{lemma}\label{lem5}
Let the process $X_n(t)$ be given by the iteration formula (\ref
{Xn2}). Then its L\'{e}vy measure and Laplace exponent are of the
following form:
%
%
\begin{equation}
\label{nu2} \nu_n(dx)=e^{-x\frac{1 }{c_1+\cdots+c_n}} \biggl(ax^{-1}+
\biggl(\frac{1
}{(c_1+\cdots+c_n)^2} \biggr) \biggr)dx,
\end{equation}
%
%
\begin{equation}
f_n(x)=a\log \Biggl(1+\sum_{1}^{n}c_ix
\Biggr)+\frac{x}{1+\sum_{1}^{n}c_ix}.
\end{equation}
Therefore, $X_n(t)$ coincides in distribution with the process
$E_{1/c}(N_{1/c}(t)+at)$, where $c=\sum_{1}^{n}c_i$.
\end{lemma}
\begin{proof}
By induction, we need to show that if (\ref{nu2}) holds at the $n$-th
step, then
%
%
\begin{equation}
\label{nun+1} \nu_{n+1}(dx)=e^{-x\frac{1 }{c_1+\cdots+c_{n+1}}} \biggl(ax^{-1}+
\biggl(\frac{1 }{(c_1+\cdots+c_{n+1})^2} \biggr) \biggr)dx.
\end{equation}
Using Lemma \ref{nu} with parameters $\alpha=\frac{1}{c_{n+1}}, \gamma
=\frac{1}{c_1+\cdots+c_{n}}$ we immediately come to~(\ref{nun+1}).
\end{proof}

We next consider the time-changed process $N_{\mu}(X_n(t))$, where
$X_n(t)$ is the process defined by the formula (\ref{Xn}), and $N_{\mu
}((t))$ is the Poisson process with parameter~$\mu$.

%
\begin{theorem}
The probabilities $p_k(t)=P \{N_\mu (X_n(t) )=k \}$
are solutions to the equation
%
%
\begin{equation}
\label{ddt} \frac{d}{dt}p_k(t)= - \sum
_{j=1}^{n}f_{n,j} \bigl(\mu (I-B )
\bigr)p_k(t)
\end{equation}
with the usual initial condition, where $f_{n,j}(u)$ are given by the
following formulas:

\noindent(i) if $\beta_i=\lambda_i=1, i=1,\xch{\ldots}{..},n$, then
\begin{eqnarray*}
f_{n,j}(u)&=&(a_{n-j}-a_{n-j-1})\log\bigl(1+(1+j)u
\bigr), \quad j=0,1,\ldots,n-1;
\\
f_{n,n}(u)&=&\frac{u}{1+nu};
\end{eqnarray*}

\noindent(ii) if $\beta_i=\beta\neq1,\lambda_i=\lambda\neq1, i=1,\xch
{\ldots}{..},n$, then
\begin{eqnarray*}
f_{n,j}(u)&=&(a_{n-j}-a_{n-j-1})\log \biggl(1+u
\frac{\gamma(\lambda,\beta
,j+1)}{\beta^{j+1}} \biggr),\quad  j=0,1,\ldots,n-1;
\\
f_{n,n}(u)&=&\frac{\lambda^n u}{\beta^n+u\gamma(\lambda,\beta,n)}.
\end{eqnarray*}
\end{theorem}

\begin{proof}
Proof follows from the formula \eqref{dtpkf1} and Theorem~\ref{thm7}.
\end{proof}

%
\begin{remark}
The equation (\ref{ddt}) can be also written in the form
%
%
\begin{equation}
\frac{d}{dt}p_{k}(t)=-f_n(\mu) p_{k}(t)+
\sum_{m=1}^{k}\frac{\mu^{m
}}{m!}p_{k-m}(t)
\int_{0}^{\infty}e^{-s\mu}s^m
\nu_n(ds), \quad  k\geq0, \, t>0,
\end{equation}
where $f_n(u)$ and $\nu_n(du)$ are given in Theorem~\ref{thm7}.
\end{remark}
%
%
\begin{remark}
In \cite{Sato,W,W2001} the process of the following form was
considered: $G(at+N(X(t)))$, where $G(t)$ is the Gamma process with
parameters $(1,1)$, that is, the process with probability density $
(\varGamma(t) )^{-1}e^{-x}x^{t-1}$ and the L\'{e}vy measure $\nu
(du)=u^{-1}e^{-u}du$ (actually, the exponential process), and $N(t)$ is
the Poisson process with parameter 1. Such a process is called the
Bessel transform of the process $X(t)$.
In the present paper we suppose that the process $G(t)$ has parameters
$(1,\beta)$ (that is, it is the exponential process with parameter
$\beta$), and the Poisson process has parameter $\lambda$ and we
consider $E_\beta (at+N_\lambda(X(t)) )$. Let us call such a
transform of the process $X(t)$, with more general parameters, the
Bessel transform as well, and denote it by
\[
B\bigl(X(t)\bigr)=B^{\beta,\lambda}\bigl(X(t)\bigr)=E \bigl(at+N\bigl(X(t)\bigr)
\bigr).
\]
Then the process $X_n(t)$, which is given by \eqref{Xn}, we can
represented as $n$-th iteration of Bessel transforms:
\[
X_n(t)=\underbrace{B_n\bigl(B_{n-1}\bigl(\ldots
B_1}_n\bigl(X_0(t)\bigr)\bigr)\bigr),
\]
where $X_0(t)=t$,
or
\[
X_n(t)=\underbrace{B_{n-1}\bigl(B_{n-2}\bigl(
\ldots B_1}_{n-1}\bigl(E \bigl(at+N(t) \bigr)\bigr)\bigr)
\bigr),
\]
where $B_i$ are Bessel transforms with parameters $\beta_i$, $\lambda_i$.

In the particular case, when $\beta_i=\lambda_i=1/c_i$, we obtain:
\[
{B_n\bigl(B_{n-1}\bigl(\ldots B_1}\bigl((t)
\bigr)\bigr)\bigr)\stackrel{d} {=}\tilde {B}(t)=E_{1/c}
\bigl(N_{1/c}(t)+at\bigr),
\]
where $c=\sum_{1}^{n}c_i$ (see Lemma~\ref{lem5}). We also have:
%
%
\begin{equation}
\label{exam1}N_\mu ({B_n^{1/c_n}
\bigl(B_{n-1}^{1/c_{n-1}}\bigl(\ldots B_1^{1/c_1}}
\bigl((t)\bigr)\bigr)\bigr)\stackrel {d} {=}N_\mu\bigl(B^{1/c}(t)
\bigr)=N_\mu\bigl(E_{1/c}\bigl(N_{1/c}(t)+at\bigr)
\bigr),
\end{equation}
where $c=\sum_{1}^{n}c_i$.

When the shift $a=0$, we recover the result stated in Remark 4 of \cite
{BS}, concerning the Poisson process with iterated time change, where
the role of time is played by the compound Poisson-exponential process
$\widetilde{E}
_{N}^{1/c}(t)=E_{1/c}(N_{1/c}(t)$ with the Laplace exponent $\tilde
{f}_{c}(u)=\frac{u}{1+cu}$. Namely, if we denote $\widetilde
{N}^{1/c}(t)=N(\widetilde{E}
_{N}^{1/c}(t))$, then the following holds:
%
%
\begin{equation}
\label{exam2} \widetilde{N}^{1/c_{1}}\bigl(\widetilde{E}_{N}^{1/c_{2}}
\bigl(\ldots\widetilde {E}%
_{N}^{1/c_{n}}(t)\ldots
\bigr)\bigr)\stackrel{d} {=}\widetilde{N}^{1/\sum
_{i=1}^{n}c_{i}}(t).
\end{equation}
The similar property with respect to iterated time change was
discovered previously for the time-changed Poisson processes where the
time is expressed by stable subordinators with Laplace exponent
$f(u)=u^\alpha$ (and called the auto-conservative property)\vadjust{\eject} in the
papers \cite{GOS,OP}. Two more cases of the iterated time change which
preserves the structure of the process are provided by the above
examples \eqref{exam1}, \eqref{exam2}.
\end{remark}


\begin{acknowledgement}[title={Acknowledgments}]
The authors are grateful to the referees for their valuable remarks and
suggestions which helped to improve the paper.
\end{acknowledgement}



\begin{thebibliography}{18}

\bibitem{ALM}
%
\begin{barticle}
\bauthor{\bsnm{Aletti}, \binits{G.}},
\bauthor{\bsnm{Leonenko}, \binits{N.N.}},
\bauthor{\bsnm{Merzbach}, \binits{E.}}:
\batitle{Fractional {P}oisson fields and martingales}.
\bjtitle{Journal of Statistical Physics}
\bid{doi={10.1007/s10955-018-1951-y}, mr={3764004}}
\end{barticle}
%
%
\OrigBibText
%
\begin{botherref}
\oauthor{\bsnm{Aletti}, \binits{G.}},
\oauthor{\bsnm{Leonenko}, \binits{N.N.}},
\oauthor{\bsnm{Merzbach}, \binits{E.}}:
Fractional {P}oisson fields and martingales.
Journal of Statistical Physics
\end{botherref}
%
\endOrigBibText
\bptok{structpyb}%
\endbibitem

\bibitem{A}
%
\begin{bbook}
\bauthor{\bsnm{Applebaum}, \binits{D.}}:
\bbtitle{L\'{e}vy {P}rocesses and {S}tochastic {C}alculus (second edition)}.
\bpublisher{Cambridge University Press}
(\byear{2009})
\bid{doi={10.1017/CBO9780511809781}, mr={2512800}}
\end{bbook}
%
\OrigBibText
%
\begin{bbook}
\bauthor{\bsnm{Applebaum}, \binits{D.}}:
\bbtitle{L\'{e}vy {P}rocesses and {S}tochastic {C}alculus (second edition)}.
\bpublisher{Cambridge University Press}
(\byear{2009})
\end{bbook}
%
\endOrigBibText
\bptok{structpyb}%
\endbibitem

\bibitem{B-D'O}
%
\begin{barticle}
\bauthor{\bsnm{Beghin}, \binits{L.}},
\bauthor{\bsnm{D'Ovidio}, \binits{M.}}:
\batitle{Fractional {P}oisson process with random drift}.
\bjtitle{Electron. J. Probab.}
\bvolume{19}
(\byear{2014})
\bid{doi={10.1214/EJP.v19-3258}, mr={3304182}}
\end{barticle}
%
%
\OrigBibText
%
\begin{botherref}
\oauthor{\bsnm{Beghin}, \binits{L.}},
\oauthor{\bsnm{D'Ovidio}, \binits{M.}}:
Fractional {P}oisson process with random drift.
Electron. J. Probab.
\textbf{19}
(2014)
\end{botherref}
%
\endOrigBibText
\bptok{structpyb}%
\endbibitem

\bibitem{B-O}
%
\begin{barticle}
\bauthor{\bsnm{Beghin}, \binits{L.}},
\bauthor{\bsnm{Orsingher}, \binits{E.}}:
\batitle{Population processes sampled at random times}.
\bjtitle{J. Stat. Phys.}
\bvolume{163},
\bfpage{1}--\blpage{21}
(\byear{2016})
\bid{doi={10.1007/s10955-016-1475-2}, mr={3472091}}
\end{barticle}
%
\OrigBibText
%
\begin{barticle}
\bauthor{\bsnm{Beghin}, \binits{L.}},
\bauthor{\bsnm{Orsingher}, \binits{E.}}:
\batitle{Population processes sampled at random times}.
\bjtitle{Journal of Statistical Physics}
\bvolume{163},
\bfpage{1}--\blpage{21}
(\byear{2016})
\end{barticle}
%
\endOrigBibText
\bptok{structpyb}%
\endbibitem

\bibitem{BS}
%
\begin{barticle}
\bauthor{\bsnm{Buchak}, \binits{K.}},
\bauthor{\bsnm{Sakhno}, \binits{L.}}:
\batitle{Compositions of {P}oisson and {G}amma processes}.
\bjtitle{Mod. Stoch.: Theory Appl.}
\bvolume{4}(\bissue{2}),
\bfpage{161}--\blpage{188}
(\byear{2017})
\bid{doi={10.15559/17-VMSTA79}, mr={3668780}}
\end{barticle}
%
\OrigBibText
%
\begin{barticle}
\bauthor{\bsnm{Buchak}, \binits{K.}},
\bauthor{\bsnm{Sakhno}, \binits{L.}}:
\batitle{Compositions of {P}oisson and {G}amma processes}.
\bjtitle{Modern Stochastics: Theory and Applications}
\bvolume{4}(\bissue{2}),
\bfpage{161}--\blpage{188}
(\byear{2017})
\end{barticle}
%
\endOrigBibText
\bptok{structpyb}%
\endbibitem

\bibitem{D}
%
\begin{barticle}
\bauthor{\bsnm{Crescenzo}, \binits{A.D.}},
\bauthor{\bsnm{Martinucci}, \binits{B.}},
\bauthor{\bsnm{Zacks}, \binits{S.}}:
\batitle{Compound {P}oisson process with a {P}oisson subordinator}.
\bjtitle{J. Appl. Probab.}
\bvolume{52}(\bissue{2}),
\bfpage{360}--\blpage{374}
(\byear{2015})
\bid{doi={10.1239/jap/1437658603}, mr={3372080}}
\end{barticle}
%
\OrigBibText
%
\begin{barticle}
\bauthor{\bsnm{Crescenzo}, \binits{A.D.}},
\bauthor{\bsnm{Martinucci}, \binits{B.}},
\bauthor{\bsnm{Zacks}, \binits{S.}}:
\batitle{Compound {P}oisson process with a {P}oisson subordinator}.
\bjtitle{Journal of Applied Probability}
\bvolume{52}(\bissue{2}),
\bfpage{360}--\blpage{374}
(\byear{2015})
\end{barticle}
%
\endOrigBibText
\bptok{structpyb}%
\endbibitem

\bibitem{GOS}
%
\begin{barticle}
\bauthor{\bsnm{Garra}, \binits{R.}},
\bauthor{\bsnm{Orsingher}, \binits{E.}},
\bauthor{\bsnm{Scavino}, \binits{M.}}:
\batitle{Some probabilistic properties of fractional point processes}.
\bjtitle{Stoch. Anal. Appl.}
\bvolume{35}(\bissue{4}),
\bfpage{701}--\blpage{718}
(\byear{2017})
\bid{doi={10.1080/07362994.2017.1308831}, mr={3651139}}
\end{barticle}
%
\OrigBibText
%
\begin{barticle}
\bauthor{\bsnm{Garra}, \binits{R.}},
\bauthor{\bsnm{Orsingher}, \binits{E.}},
\bauthor{\bsnm{Scavino}, \binits{M.}}:
\batitle{Some probabilistic properties of fractional point processes}.
\bjtitle{Stochastic Analysis and Application}
\bvolume{35}(\bissue{4}),
\bfpage{701}--\blpage{718}
(\byear{2017})
\end{barticle}
%
\endOrigBibText
\bptok{structpyb}%
\endbibitem

\bibitem{HMS}
%
\begin{barticle}
\bauthor{\bsnm{Haubold}, \binits{H.J.}},
\bauthor{\bsnm{Mathai}, \binits{A.M.}},
\bauthor{\bsnm{Saxena}, \binits{R.K.}}:
\batitle{Mittag-{L}effler functions and their applications}.
\bjtitle{Journal of Applied Mathematics}
\bvolume{51}
(\byear{2011})
\bid{doi={10.1155/2011/298628}, mr={2800586}}
\end{barticle}
%
%
\OrigBibText
%
\begin{botherref}
\oauthor{\bsnm{Haubold}, \binits{H.J.}},
\oauthor{\bsnm{Mathai}, \binits{A.M.}},
\oauthor{\bsnm{Saxena}, \binits{R.K.}}:
Mittag-{L}effler functions and their applications.
Journal of Applied Mathematics,
51
(2011)
\end{botherref}
%
\endOrigBibText
\bptok{structpyb}%
\endbibitem

\bibitem{KS}
%
\begin{barticle}
\bauthor{\bsnm{Kobylych}, \binits{K.}},
\bauthor{\bsnm{Sakhno}, \binits{L.}}:
\batitle{Point processes subordinated to compound {P}oisson processes}.
\bjtitle{Theory Probab. Math. Stat.}
\bvolume{94},
\bfpage{85}--\blpage{92}
(\byear{2016}).
\bcomment{(in Ukrainian); English translation in Theory of
Probability and Mathematical Statistics \texbold{94}, 89--96 (2017)}
\bid{doi={10.1090/tpms/1011}, mr={3553456}}
\end{barticle}
%
\OrigBibText
%
\begin{barticle}
\bauthor{\bsnm{Kobylych}, \binits{K.}},
\bauthor{\bsnm{Sakhno}, \binits{L.}}:
\batitle{Point processes subordinated to compound {P}oisson processes}.
\bjtitle{Theory of Probability and Mathematical Statistics}
\bvolume{94},
\bfpage{85}--\blpage{92}
(\byear{2016}).
\bcomment{(in Ukrainian); English translation in Theory of
Probability and Mathematical Statistics 94, 89--96 (2017) }
\end{barticle}
%
\endOrigBibText
\bptok{structpyb}%
\endbibitem

\bibitem{LST}
%
\begin{barticle}
\bauthor{\bsnm{Leonenko}, \binits{N.}},
\bauthor{\bsnm{Scalas}, \binits{E.}},
\bauthor{\bsnm{Trinh}, \binits{M.}}:
\batitle{The fractional non-homogeneous {P}oisson process}.
\bjtitle{Stat. Probab. Lett.}
\bvolume{120},
\bfpage{147}--\blpage{156}
(\byear{2017})
\bid{doi={10.1016/\\j.spl.2016.09.024}, mr={3567934}}
\end{barticle}
%
\OrigBibText
%
\begin{barticle}
\bauthor{\bsnm{Leonenko}, \binits{N.}},
\bauthor{\bsnm{Scalas}, \binits{E.}},
\bauthor{\bsnm{Trinh}, \binits{M.}}:
\batitle{The fractional non-homogeneous {P}oisson process}.
\bjtitle{Statistics and Probability Letters}
\bvolume{120},
\bfpage{147}--\blpage{156}
(\byear{2017})
\end{barticle}
%
\endOrigBibText
\bptok{structpyb}%
\endbibitem

\bibitem{OP2}
%
\begin{barticle}
\bauthor{\bsnm{Orsingher}, \binits{E.}},
\bauthor{\bsnm{Polito}, \binits{F.}}:
\batitle{Compositions, random sums and continued random fractions of {P}oisson
and fractional {P}oisson processes}.
\bjtitle{J. Stat. Phys.}
\bvolume{148},
\bfpage{233}--\blpage{249}
(\byear{2012})
\bid{doi={10.1007/s10955-012-0534-6}, mr={2966360}}
\end{barticle}
%
\OrigBibText
%
\begin{barticle}
\bauthor{\bsnm{Orsingher}, \binits{E.}},
\bauthor{\bsnm{Polito}, \binits{F.}}:
\batitle{Compositions, random sums and continued random fractions of {P}oisson
and fractional {P}oisson processes}.
\bjtitle{J. Stat. Phys.}
\bvolume{148},
\bfpage{233}--\blpage{249}
(\byear{2012})
\end{barticle}
%
\endOrigBibText
\bptok{structpyb}%
\endbibitem

\bibitem{OP}
%
\begin{barticle}
\bauthor{\bsnm{Orsingher}, \binits{E.}},
\bauthor{\bsnm{Polito}, \binits{F.}}:
\batitle{The space-fractional {P}oisson process}.
\bjtitle{Stat. Probab. Lett.}
\bvolume{82},
\bfpage{852}--\blpage{858}
(\byear{2012})
\bid{doi={10.1016/j.spl.2011.12.018}, mr={2899530}}
\end{barticle}
%
\OrigBibText
%
\begin{barticle}
\bauthor{\bsnm{Orsingher}, \binits{E.}},
\bauthor{\bsnm{Polito}, \binits{F.}}:
\batitle{The space-fractional {P}oisson process}.
\bjtitle{Statistics and Probability Letters}
\bvolume{82},
\bfpage{852}--\blpage{858}
(\byear{2012})
\end{barticle}
%
\endOrigBibText
\bptok{structpyb}%
\endbibitem

\bibitem{OP2013}
%
\begin{barticle}
\bauthor{\bsnm{Orsingher}, \binits{E.}},
\bauthor{\bsnm{Polito}, \binits{F.}}:
\batitle{On the integral of fractional {P}oisson processes}.
\bjtitle{Stat. Probab. Lett.}
\bvolume{83}(\bissue{4}),
\bfpage{1006}--\blpage{1017}
(\byear{2013})
\bid{doi={10.1016/j.spl.2012.12.016}, mr={3041370}}
\end{barticle}
%
\OrigBibText
%
\begin{barticle}
\bauthor{\bsnm{Orsingher}, \binits{E.}},
\bauthor{\bsnm{Polito}, \binits{F.}}:
\batitle{On the integral of fractional {P}oisson processes}.
\bjtitle{Statistics and Probability Letters}
\bvolume{83}(\bissue{4}),
\bfpage{1006}--\blpage{1017}
(\byear{2013})
\end{barticle}
%
\endOrigBibText
\bptok{structpyb}%
\endbibitem

\bibitem{OT}
%
\begin{barticle}
\bauthor{\bsnm{Orsingher}, \binits{E.}},
\bauthor{\bsnm{Toaldo}, \binits{B.}}:
\batitle{Counting processes with {B}ern\v{s}tein intertimes and random jumps}.
\bjtitle{J. Appl. Probab.}
\bvolume{52},
\bfpage{1028}--\blpage{1044}
(\byear{2015})
\bid{doi={10.1239/\\jap/1450802751}, mr={3439170}}
\end{barticle}
%
\OrigBibText
%
\begin{barticle}
\bauthor{\bsnm{Orsingher}, \binits{E.}},
\bauthor{\bsnm{Toaldo}, \binits{B.}}:
\batitle{Counting processes with {B}ern\u{s}tein intertimes and random jumps}.
\bjtitle{Journal of Applied Probability}
\bvolume{52},
\bfpage{1028}--\blpage{1044}
(\byear{2015})
\end{barticle}
%
\endOrigBibText
\bptok{structpyb}%
\endbibitem

\bibitem{Sato}
%
\begin{bbook}
\bauthor{\bsnm{Sato}, \binits{K.}}:
\bbtitle{L\'{e}vy Processes and Infinitely Divisible Distributions}.
\bpublisher{Cambridge University Press}
(\byear{1999})
\bid{mr={1739520}}
\end{bbook}
%
\OrigBibText
%
\begin{bbook}
\bauthor{\bsnm{Sato}, \binits{K.}}:
\bbtitle{L\'{e}vy Processes and Infinitely Divisible Distributions}.
\bpublisher{Cambridge University Press}
(\byear{1999})
\end{bbook}
%
\endOrigBibText
\bptok{structpyb}%
\endbibitem

\bibitem{Sn}
%
\begin{bbook}
\bauthor{\bsnm{Sneddon}, \binits{I.N.}}:
\bbtitle{Special functions of mathematical physics and chemistry}.
\bpublisher{Oliver and Boyd}, \blocation{Edinburgh}
(\byear{1956})
\bid{mr={0080170}}
\end{bbook}
%
%
\OrigBibText
%
\begin{botherref}
\oauthor{\bsnm{Sneddon}, \binits{I.N.}}:
Special functions of mathematical physics and chemistry.
Oliver and Boyd, Edinburgh
(1956)
\end{botherref}
%
\endOrigBibText
\bptok{structpyb}%
\endbibitem

\bibitem{W2001}
%
\begin{botherref}
\oauthor{\bsnm{Watanabe}, \binits{T.}}:
Temporal change in distributional properties of {L}\'{e}vy processes
\end{botherref}
%
\OrigBibText
%
\begin{botherref}
\oauthor{\bsnm{Watanabe}, \binits{T.}}:
Temporal change in distributional properties of {L}\'{e}vy processes
\end{botherref}
%
\endOrigBibText
\bptok{structpyb}%
\endbibitem

\bibitem{W}
%
\begin{barticle}
\bauthor{\bsnm{Watanabe}, \binits{T.}}:
\batitle{On {B}essel transforms of multimodal increasing {L}\'{e}vy processes}.
\bjtitle{Jpn. J. Math.}
\bvolume{25}(\bissue{2}),
\bfpage{227}--\blpage{256}
(\byear{1999})
\bid{doi={10.4099/math1924.25.227}, mr={1735462}}
\end{barticle}
%
\OrigBibText
%
\begin{barticle}
\bauthor{\bsnm{Watanabe}, \binits{T.}}:
\batitle{On {B}essel transforms of multimodal increasing {L}\'{e}vy processes}.
\bjtitle{Japan. J. Math.}
\bvolume{25}(\bissue{2}),
\bfpage{227}--\blpage{256}
(\byear{1999})
\end{barticle}
%
\endOrigBibText
\bptok{structpyb}%
\endbibitem

\end{thebibliography}
\end{document}